\documentclass[reqno,12pt]{amsart}
\usepackage{amsmath,amsthm,amsfonts,amssymb,ifpdf}
\usepackage{stmaryrd}

\usepackage{amssymb}
\usepackage{amsmath}
\usepackage{amsthm}
\usepackage{graphicx}
\usepackage[all]{xy}
\usepackage{enumerate}
\usepackage{tikz-cd}
\newcommand{\RomanNumeralCaps}[1]
    {\MakeUppercase{\romannumeral #1}}
\theoremstyle{plain} 
\newtheorem{theorem}{Theorem}[section]
\newtheorem{proposition}[theorem]{Proposition}
\newtheorem{lemma}[theorem]{Lemma}
\newtheorem{corollary}[theorem]{Corollary}

\theoremstyle{definition} \newtheorem{definition}[theorem]{Definition}

\theoremstyle{remark} \newtheorem{remark}[theorem]{Remark}
\newtheorem{notation}[theorem]{Notation}

\def\Hodge{{\rm Hodge}}
\def\deg{{\rm deg}}
\def\point{{\rm point}}
\def\Pic{{\rm Pic}}
\def\PGL{{\rm PGL}}
\def\GL{{\rm GL}}
\usepackage[all]{xy}

\ifpdf
  \usepackage[
    pdftex,
    colorlinks,%
    linkcolor=blue,citecolor=red,urlcolor=blue,
    hyperindex,%
    plainpages=false,%
    bookmarksopen,%
    bookmarksnumbered%
  ]{hyperref} 
 \usepackage{thumbpdf}
\else
  \usepackage{hyperref}
\fi

\newcommand\blfootnote[1]{%
  \begingroup
  \renewcommand\thefootnote{}\footnote{#1}%
  \addtocounter{footnote}{-1}%
  \endgroup
}

\newcommand{\ncom}{\newcommand}
\ncom{\mylabel}[1]{{\rm (#1)}\label{#1}}
\ncom{\Hom}{{\textit{Hom}}}
\ncom{\eop}{{\hfill $\Box$}}
\begin{document}
\baselineskip=16pt

\setcounter{tocdepth}{1}

\title[Tautological classes]{Tautological algebra of the moduli stack of semistable bundles of rank \MakeLowercase{2} on a general curve}

\author[C. Gangopadhyay]{Chandranandan Gangopadhyay}
\author[J. N. Iyer]{Jaya NN  Iyer}

\author[A. Mukherjee]{Arijit Mukherjee}

\address{Department of Mathematics, Shiv Nadar University, NH91, Tehsil Dadri, Greater Noida, Uttar
Pradesh 201314, India}
\email{chandranandan.g@snu.edu.in}

\address{The Institute of Mathematical Sciences, C.I.T
Campus, Taramani, Chennai 600113, India}
\email{jniyer@imsc.res.in}

\address{Department of Mathematics, Indian Institute of Technology Madras, IIT P.O., Chennai, Tamil Nadu - 600 036, India.}
\email{mukherjee90.arijit@gmail.com}

\blfootnote{Mathematics Subject Classification 2020 : 14D20, 14D23, 14H40, 14H51, 14H60.}
\blfootnote{Key words and phrases : Moduli stacks, semistable bundles, tautological classes, Chow groups.}

\begin{abstract}
Our aim is to determine the tautological algebra generated by the cohomology classes of the Brill-Noether loci in the rational cohomology of  the moduli stack $\mathcal{U}_C(n,d)$ of semistable bundles of rank $n$ and degree $d$. We show that for a general smooth projective curve $C$ of genus $g\geq 2$, $d=2g-2$, the tautological algebra of $ \mathcal{U}_C(2,2g-2)$ (resp. the moduli stack  $\mathcal{SU}_C(2,\mathcal{L})$ of semistable bundles of rank $2$ and determinant $\mathcal{L}$ with $\deg(\mathcal{L})=2g-2$) is generated by the divisor classes (resp. the class of the Theta divisor $\Theta$).  This is previously known in rank one situation, called the (classical) Porteous formula.
\end{abstract}
\maketitle

\tableofcontents

\section{Introduction}
\label{sec:Introduction}

Given a moduli space or stack, the cohomology ring is  defined, and determining the generators and structure of the ring is a wide area of interest and study. A more manageable area of study is to look at a class of naturally defined geometric loci, and look at the $\mathbb{Q}$-subalgebra defined by their cohomology classes. This is usually called the $\textit{Tautological algebra}$. These have been investigated on Jacobian varieties (cf. \cite{ACGH}), moduli space of curves (cf. \cite{Fb}), moduli space of abelian varieties (cf. \cite{vdG}), to cite a few examples. In this paper, we would like to consider this area of study on the moduli space of vector bundles on a curve $C$.  

Suppose $C$ is a complex smooth projective curve of genus $g$. The Jacobian variety $J_d(C)$ parametrises isomorphism classes of degree $d$ line bundles on $C$. The classical Brill-Noether subvarieties $W^r_d$ of $J_d(C)$ parametrise line bundles with at least $r+1$ linearly independent sections.  The  questions on  non-emptiness, dimension, and irreducibility of the loci have been classically studied (cf. \cite[Chapter \RomanNumeralCaps{4}]{ACGH}, \cite{Gr-Ha2} and \cite{Su}).
  In other direction, the classical Poincar\'e formula expresses the cohomological classes of $W^0_i$, in terms of the Theta divisor  on $ J(C)$:
$$
[W^0_i] = \frac{1}{(g-i)!} [\Theta]^{g-i} \in H^*(J(C),\mathbb{Q}).
$$
Here we identify $J_d(C)\cong J_0(C)=J(C)$ (cf. \cite[Chapter 1, \S5; p. 25]{ACGH}).

A similar formula holds amongst the cohomology classes of  $W^r_d$, for varying $r$ in $J_d(C)$, for a general smooth curve $C$ (see Theorem \ref{dimension of BN loci for general curve}) :
\begin{equation*}\label{porteousformula}
	[W^r_d]=\prod\limits^r_{\alpha=0}\dfrac{\alpha!}{(g-d+r+\alpha)!}.[\Theta_d]^{g-\rho}.
\end{equation*}
Here $\rho=\dim(W^r_d)$.  When $\rho=0$, the above formula is known as the Castelnuovo formula or the Porteous formula, in general. 

For $n,d\geq 1$ we denote by $\mathcal{U}_C(n,d)$ the moduli stack of semistable vector bundles of rank $n$ and degree $d$.  For $\mathcal L\in {\rm Pic}(C)$ we denote by $\mathcal{SU}_C(n,\mathcal{L})$  the moduli stack of semistable vector bundles of rank $n$ and determinant $\mathcal{L}$.  Their coarse moduli spaces 
${U}_C(n,d)$ and ${SU}_C(n,\mathcal{L})$ have been widely studied. The Brill-Noether loci have been similarly defined and investigated, (cf. \cite{Bi}, \cite{Bi2}, \cite{BGN}, \cite{M1}, \cite{M2} and  \cite{Su}).  Further developments on  the subject of Brill-Noether loci can be found in \cite{La-Ne-St}, \cite{La-Ne-Pr}, \cite{La-Ne3}, \cite{La-Ne4} and \cite{La-Ne5}. However in this paper, we won't be pursuing this line of study.
 
Some questions on the cohomology classes have been raised by C. S. Seshadri  and N. Sundaram in \cite[p. 176]{Su}.  As pointed out by P. Newstead, the virtual cohomology classes can be worked out in terms of the known generators of the cohomology of moduli space when $n$ and $d$ are coprime,  using the determinantal structure of the Brill-Noether loci.  
 In general the question on finding relations amongst the classes are open, except in rank one situation. In this paper,  we illustrate a method to address this question. We consider  rank two situation, and when the degree $d=2g-2$, for any $g\geq 2$.  

Our aim in this paper is to obtain a Poincar\'e type relation amongst the cohomological classes of the geometric Brill-Noether loci, when the curve $C$ is general.  As a first step, in \cite{Mk}, when $g=1$ the relations were found to be similar to the Poincar\'e relations.
The moduli spaces ${U}_C(2, 2g-2)$ and ${SU}_C(2,\mathcal{L})$ are singular varieties, and  the cohomology class of a Brill-Noether locus is well-defined
in the cohomology of the  moduli stack $H^{*}(\mathcal{U}_C(2, 2g-2),\mathbb{Q})$ (resp.
$ H^{*}(\mathcal{SU}_C(2,\mathcal{L}),\mathbb{Q})$). See  \S \ref{CohBN}).  

We show the following:
\begin{theorem}\label{mainThm} Suppose $C$ is a general smooth projective curve of genus $g$, and $g\geq 2$.
The cohomology class of a Brill-Noether locus on the moduli stack $\mathcal{U}_C(2,2(g-1))$ can be expressed as a polynomial on the divisor classes, with rational coefficients.
\end{theorem}
See Theorem \ref{MT}.
  
Similarly, consider the moduli stack $\mathcal{SU}_C(2,\mathcal{L})$ of semistable bundles of rank $r$ and fixed determinant $\mathcal{L}$ of degree $2g-2$ on $C$.  Then we obtain the following:
\begin{corollary}
The cohomology class of a Brill-Noether locus  $\widetilde{W^{r,\mathcal{L}}_{2,2(g-1)}}$ in the moduli stack $\mathcal{SU}_C(2,\mathcal{L})$ is a polynomial expression on the class of the Theta divisor, with rational coefficients. In other words, the tautological algebra generated by the Brill-Noether loci is generated by the class of the theta divisor. 
\end{corollary}
See Corollary \ref{MT2}. This is close to the Porteous formula.

The key idea is to relate the Brill-Noether loci on the moduli space with the Brill-Noether loci on the Jacobian variety of a general spectral curve. 
We utilise the rational map obtained in \cite{BNR} from the Jacobian of a general spectral curve to the moduli space ${U}_C(2,2(g-1))$. Another aspect is to note that the Hodge conjecture holds for the Jacobian of a general two sheeted spectral curve, via a computation of the Mumford-Tate group (cf.  \cite{B}). 
In particular, we note:
\begin{theorem}\label{JSpecthm}
	The tautological algebra of the Jacobian  
	$J(\widetilde{C})$ of a general $2$-sheeted spectral curve $\pi:\widetilde{C}\rightarrow C$ is generated by the divisor classes.
\end{theorem}
See Theorem \ref{Bis}. Note that this is primarily the reason we deal with the rank $2$ case. For the general case, we need the same result for $n$-sheeted covers.

We also use the fact that the moduli stacks are quotient stacks (cf. \cite{Gomez}). Hence, their cohomology is the equivariant cohomology of  (an open subset of) Quot scheme. The Brill-Noether loci are equivariant for the group action $\GL(N)$, see Lemma \ref{GLn-invariantlemma}.  
This enables us to define the Brill-Noether classes in the cohomology of the moduli stack. See \S \ref{CohBN}.

For moduli stacks $\mathcal{U}_C(n,d)$, when $n>2$ and $d\neq 2g-2$, the same proofs and techniques hold. However, to obtain the final conclusion, we need to have that the Hodge conjecture (resp. Theorem \ref{JSpecthm}) holds for Jacobian of a higher  degree general spectral curve
(cf. \cite{Ar}, for unramified coverings). The proofs employed in Theorem \ref{mainThm} will then also hold for higher rank moduli spaces.  For degree $d\neq 2g-2$, the divisor classes on the Jacobian of spectral curve  descend on the moduli space (see Lemma \ref{descend}), in several other cases too. (This was mentioned by Newstead). 

We believe that the techniques used in this paper can be utilised in other contexts too. To our knowledge, no results on Brill-Noether loci are known on moduli stacks. Using the group-equivariance of the Brill-Noether loci, it might be possible to study them further. 


\section{Notations}

All the varieties are defined over complex numbers.
\begin{enumerate}
\item $\mathcal{U}_{C}(n,d)$ denotes the moduli stack of semistable bundles of rank $r$ and degree $d$ over $C$. It admits a good moduli space and we denote it by $U_C(n,d)$.
\item $\mathcal{SU}_{C}(n,\mathcal{L})$ denotes the moduli stack of semistable bundles of rank $r$ and fixed determinant $\mathcal{L}$ of degree $d$ over $C$. $SU_C(n,\mathcal{L})$ denotes its good moduli space.
\item $J(C)$ is the Jacobian variety of isomorphism classes of line bundles of degree $0$ over $C$.
\item $J_{d}(C)$ is the isomorphism classes of line bundles of degree $d$ over $C$.
\item $\mathcal{O}(D)$ denotes the line bundle corresponding to a divisor $D$ on $C$.
\item Given a closed subvariety $W\subset X$, $[W]$ denotes the cohomology class in the integral or rational cohomology group of $X$.
\end{enumerate}
\if
1) $\mathcal{U}_{C}(n,d)$ denotes the moduli space of strong equivalence
classes of semistable bundles of rank $r$ and degree $d$ over $C$.\\
2) $\mathcal{SU}_{C}(n,\mathcal{L})$ denotes the moduli space of strong equivalence
classes of semistable bundles of rank $r$ and fixed determinant $\mathcal{L}$ of degree $d$ over $C$.\\
3) $J(C)$ is the Jacobian variety of isomorphism classes of line bundles of
degree $0$ over $C$.\\
4) $J_{d}(C)$ is the isomorphism classes of line bundles of degree $d$ over $C$.\\
5) $\mathcal{O}(D)$ denotes the line bundle corresponding to a divisor $D$ on $C$.\\
6) Given a closed subvariety $W\subset X$, $[W]$ denotes the cohomology class in the integral or rational cohomology group of $X$.\\
\fi

\section{Spectral curves and Moduli Spaces} \label{spectral}
\label{sec:Spectral curves}

In this section we recall the construction of spectral curve from \cite{BNR} which will be needed in this paper.

\subsection{Spectral curve}
 Let $C$ be a smooth projective curve of genus $g\geq 2$ defined over complex numbers.  
Fix $n\geq 1$. Let $L$ be a line bundle on $C$ and $s=(s_k)$ be sections of $L^k$ for $k=1,2,\cdots,n$.  Let $\pi:\mathbb{P}(\mathcal{O}_C\oplus L^{\ast})\rightarrow C$ be the natural projection map and $\mathcal{O}(1)$ be the relatively ample bundle.  Then $\pi_{\ast}(\mathcal{O}(1))$ is naturally isomorphic to $\mathcal{O}_C \oplus L^{\ast}$ and therefore has a canonical section.  This provides a section of $\mathcal{O}(1)$ denoted by $y$.  By projection formula we have:
	\begin{equation*}
	\pi_{\ast}(\pi^{\ast}L\otimes \mathcal{O}(1))\cong L\otimes \pi_{\ast}(\mathcal{O}(1))\cong L\otimes (\mathcal{O}_C\oplus L^{\ast})=L\oplus \mathcal{O}_C.
	\end{equation*}
Therefore $\pi_{\ast}(\pi^{\ast}L\otimes \mathcal{O}(1))$ also has a canonical section and we denote the corresponding section of $\pi^{\ast}L\otimes \mathcal{O}(1)$ by $x$.  Consider the section
\begin{equation}\label{SC}
x^n+(\pi^{\ast}s_1)yx^{n-1}+\cdots+(\pi^{\ast}s_n)y^n
\end{equation}  
of $\pi^{\ast}L^n\otimes \mathcal{O}(n)$.  Zero scheme of this section is a subscheme of $\mathbb{P}(\mathcal{O}_C\oplus L^{\ast})$ and is called spectral curve of the given curve $C$ and is denoted by $\widetilde{C}_{s}$ or $\widetilde{C}$ in short.  Let $\pi:\widetilde{C}\rightarrow C$ be the restriction of the natural projection  $\pi:\mathbb{P}(\mathcal{O}_C\oplus L^{\ast})\rightarrow C$.  It can be checked that $\pi:\widetilde{C}\rightarrow C$ is finite and its fiber over any point $c\in C$ is a subscheme of $\mathbb{P}^1$ given by
\begin{equation*}
x^n+a_1yx^{n-1}+\cdots+a_ny^n=0,
\end{equation*}
where $(x,y)$ is a homogeneous co-ordinate system and $a_i$ is the value of $s_i$ at $c$.\\
Let $\widetilde{g}$ be the genus of $\widetilde{C}$.  As $\pi_{\ast}(\mathcal{O}_{\widetilde{C}})\cong \mathcal{O}_C\oplus L^{-1}\oplus \cdots \oplus L^{-(n-1)}$, we have the following relation between genus $\widetilde{g}$ of the spectral curve $\widetilde{C}$ and genus $g$ of $C$ using Riemann-Roch theorem
\begin{equation*}
1-\widetilde{g}=\chi(\widetilde{C},\mathcal{O}_{\widetilde{C}})=\chi(C,\pi_{\ast}(\mathcal{O}_{\widetilde{C}}))=\sum_{i=0}^{n-1}\chi(C,L^{-i})=-\deg(L)\cdot\frac{ n(n-1)}{2}+n(1-g).
\end{equation*}
Hence we have:
\begin{equation}\label{E1}
\widetilde{g}=\deg(L)\cdot\frac{ n(n-1)}{2}+n(g-1)+1.
\end{equation}
Moreover if we take the line bundle $L$ to be of degree $2g-2$, say the canonical line bundle $K_C$ for example, then from (\ref{E1}) the genus $\widetilde{g}$ of the corresponding spectral curve $\widetilde{C}$ is given by:
\begin{equation}\label{E2}
\widetilde{g}=n^2(g-1)+1= \dim({U}_{C}(n,d)).
\end{equation} 


\subsection{Spectral curve and moduli space of semistable bundles}

Here we relate the spectral curve $\widetilde{C}$ with the moduli space of semistable bundles of fixed rank and degree over $C$.  Consider the following theorem.
\begin{theorem}\label{T1}
	Let $C$ be any curve and $L$ any line bundle on $C$.  Let $(s)=((s_i))\in \Gamma(L)\oplus \Gamma(L^2)\oplus \cdots \oplus \Gamma(L^n)$ be so chosen such that the corresponding spectral curve $\widetilde{C}_s$ is integral, smooth and non-empty.  Then there is a bijective correspondence between isomorphism classes of line bundles on $\widetilde{C}_s$ and isomorphism classes of pairs $(E,\phi)$ where $E$ is a vector bundle of rank $n$ and $\phi:E\rightarrow L\otimes E$ a homomorphism with characteristic coefficients $s_i$.
\end{theorem}
\begin{proof}
	See {\cite[Proposition 3.6, Remark 3.1,3.5 and 3.8; p.~172-174]{BNR}}.
\end{proof}
\begin{remark}
    From the stack theoretic perspective, Theorem 3.1 may be interpreted as
the following: There is a morphism from $J_{\delta}(\widetilde C)\to \mathcal{M}_{C}(n,d)$, where $\delta=d+(n^2-n)(g-1)$ (cf. \eqref{E3}) and 
$\mathcal{M}_C(n,d)$ is the stack of all vector bundles on C of rank $n$ and degree $d$.  
\end{remark}
 
Let $n$ be any positive integer. Then following the construction of spectral curve, by Theorem \ref{T1} we get a smooth, irreducible curve $\widetilde{C}$ and an $n$-sheeted branched covering $\pi:\widetilde{C}\rightarrow C$ such that a general $E\in {U}_{C}(n,d)$ is the direct image $\pi_{\ast}l$ of a $l\in J_{\delta}(\widetilde{C})$.  The relation between $\delta$ and $d$ can be calculated as follows (cf. \cite[p.~332]{BT}).
By the Leray spectral sequence we have:
\begin{equation}\label{LSS}
H^i(\widetilde{C},l)=H^i(C,\pi_{\ast}l)
\end{equation}
for all $i$.  Hence we have:
\begin{equation*}
\chi(\widetilde{C},l)=\chi(C,\pi_{\ast}l)=\chi(C,E).
\end{equation*}
So  by Riemann-Roch theorem we get,
\begin{equation*}
\begin{split}
& \chi(\widetilde{C},l)=\chi(C,E)\\
\Rightarrow & \delta-(\widetilde{g}-1)=d-n(g-1)\\
\Rightarrow & \delta =d+(\widetilde{g}-1)-n(g-1).
\end{split}
\end{equation*}
Therefore by (\ref{E2}) we get the following relation between $\delta(=\deg(l))$ and $d(=\deg(E))$ :
\begin{equation}\label{E3}
\delta=d+(n^2-n)(g-1).
\end{equation}
As direct image of a line bundle is not necessarily semistable, the map 
\begin{equation*}
\pi_{\ast}:J_{\delta}(\widetilde{C})\dashrightarrow {U}_{C}(n,d)
\end{equation*}
is only a rational map.
Let us denote by $J^{ss}$ the semistable locus  of $J_{\delta}(\widetilde{C})$ defined as:
\begin{equation*}
J^{ss}:=\{l\in J_{\delta}(\widetilde{C})\mid \pi_{\ast}l\in {U}_{C}(n,d)\}.
\end{equation*}
Then $J^{ss}$ is a Zariski open subset of $J_{\delta}(\widetilde{C})$ and the map 
\begin{equation}\label{E39}
\pi_{\ast}:J^{ss}\rightarrow {U}_{C}(n,d)
\end{equation}
is a regular dominant map (cf. \cite[Theorem 1; p.~169]{BNR}).
Moreover, the following theorem shows that the map $\pi_{\ast}$ is a generically finite map.
\begin{theorem}\label{T2}
	The map $\pi_{\ast}:J^{ss}\rightarrow {U}_{C}(n,d)$ is dominant and is of degree $2^{3g-3} 3^{5g-5}\cdots n^{(2n-1)(g-1)}$.
\end{theorem}
\begin{proof}
See {\cite[Remark 5.4; p.~177]{BNR}}.
\end{proof}

\section{Jacobian of a spectral curve and cycle class maps}\label{Prym}

In this section we consider the moduli space ${U}_{C}(2,d)$.  For a general $E\in {U}_{C}(2,d)$ we get a spectral curve $\pi:\widetilde{C}\rightarrow C$ where the map $\pi$ is a $2$-sheeted branched covering.  Let $n$ be the number of branch points.  Then by Riemann-Hurwitz formula we get:
\begin{equation}\label{E4}
\widetilde{g}=\frac{n}{2}+2g-1.
\end{equation}
Also we have from (\ref{E2}):
\begin{equation}\label{E5}
\widetilde{g}=4g-3.
\end{equation}
Therefore from (\ref{E4}) and (\ref{E5}) we get, 
\begin{equation*}
n=4g-4\neq 0 ~as ~ g\geq 2
\end{equation*}
that is, $\pi:\widetilde{C}\rightarrow C$ is ramified with $4g-4$ branch points.
Now we have the following lemma.
\begin{lemma}\label{L1}
	The map $\pi^{\ast}:J(C)\rightarrow J(\widetilde{C})$ is injective.
\end{lemma}
\begin{proof}
	See {\cite[Lemma; p.~332]{Mum}}.
\end{proof}
Consider the following Norm map, denoted by $\text{Nm}(\pi)$, associated to the map $\pi:\widetilde{C}\rightarrow C$.
\begin{equation*}
\begin{split}
\text{Nm}(\pi): J(\widetilde{C}) &\rightarrow J(C)\\
\sum_{i=1}^m n_i \widetilde{x}_i &\mapsto \sum_{i=1}^m n_i \pi(\widetilde{x}_i).
\end{split}
\end{equation*}
The identity component of $\text{Ker}(\text{Nm}(\pi))$ is defined to be the Prym variety associated to the covering $\pi:\widetilde{C}\rightarrow C$.  
Moreover we have, $J(\widetilde{C})\cong J(C)+P$ upto isogeny, where $P$ is the Prym variety associated to the covering $\pi:\widetilde{C}\rightarrow C$. 


\subsection{The Hodge (p,p)-conjecture}

Let $X$ be a smooth projective variety over complex numbers.  Let $CH^p(X)$ be the $p$-th graded piece of the Chow ring $CH^{\ast}(X)$,(cf. \cite[Chapter 1; p.~6]{Fulton}).
There is a cycle class map as follows.
\begin{equation}\label{E40}
\begin{split}
cl: CH^p(X) & \rightarrow H^{2p}(X,\mathbb{Z}).
\end{split}
\end{equation}
  Let $CH^{\ast}(X)\otimes \mathbb{Q}$ be denoted by $CH^{\ast}(X)_\mathbb{Q}$ and $j_\mathbb{Q}:H^{2p}(X,\mathbb{Q})\rightarrow H^{2p}(X,\mathbb{C})$ be the natural map.  Then the subspace of Hodge classes of $H^{2p}(X,\mathbb{Q})$,  denoted by $H^{2p}_{\Hodge}(X)$, is defined as:
\begin{equation*}\label{E42}
H^{2p}_{\Hodge}(X):= H^{2p}(X,\mathbb{Q})\cap j_{\mathbb{Q}}^{-1}(H^{p,p}(X)) .
\end{equation*} 
Consider the cycle class map $cl:CH^p(X)_\mathbb{Q} \rightarrow H^{2p}(X,\mathbb{Q})$ defined similarly as in (\ref{E40}).  The image of this map is denoted by $H^{2p}(X,\mathbb{Q})_{alg}$ and the elements in  $H^{2p}(X,\mathbb{Q})_{alg}$ are called rational algebraic classes.  The Hodge $(p,p)$-conjecture asserts the following:\\
\textbf{Hodge conjecture :}
\begin{equation*}
H^{2p}(X,\mathbb{Q})_{alg}=H^{2p}_{\Hodge}(X)\;,
\end{equation*}
that is, any rational algebraic class is a Hodge class and vice versa.

\begin{remark}\label{R1}
Hodge $(p,p)$-conjecture is trivially true for $p=0$. For $p=1$ see \cite{L}.
\end{remark}

\subsection{The Hodge conjecture for general abelian varieties}

Let $X$ be an abelian variety.  The subring of $H_{\Hodge}^{2\ast}(X)$ generated by $H_{\Hodge}^0(X)$ and $H_{\Hodge}^2(X)$ is denoted by $D^{\ast}(X)$.  The cycle classes in $D^{\ast}(X)$   are all algebraic by Remark \ref{R1}.  Let $D^p(X)$ be the $p$th graded piece of $D^{\ast}(X)$.  In particular, the Hodge $(p,p)$-conjecture is true if 
\begin{equation*}
D^p(X)=H^{2p}_{\Hodge}(X).
\end{equation*}
 Using degeneration techniques, Hodge conjecture is known to hold for a general polarised Jacobian variety with Theta divisor $\Theta$ as a polarisation, see {\cite[Theorem 17.5.1; p.~561]{LB}}. (See cf. \cite{M}, for general polarized abelian varieties).

We will note in \S \ref{JSpec}, that the Hodge conjecture holds for the Jacobian of a general spectral curve. This will be crucial to obtain the relations amongst the Brill-Noether loci on the moduli stacks. 

\section{Equivariant Cohomology and Equivariant Chow groups}

In this section, we recall some preliminary facts on the equivariant
 groups for a smooth variety $X$ of
 dimension $d$, which is equipped with an action by a linear reductive
 algebraic group $G$.
The equivariant groups and their properties that we recall below were
 defined by Borel, Totaro, Edidin-Graham, Fulton
(cf. \cite{Borel},\cite{Totaro},\cite{EdidinGraham}, \cite{Fulton} and \cite{Edidin}). 

\subsection{Equivariant cohomology $H^i_G(X,\mathbb{Z})$ of $X$}

Suppose $X$ is a variety with an action on the left by an algebraic
 group $G$. Borel defined the equivariant cohomology $H^*_G(X)$ as
 follows.
There is a contractible space $EG$ on which $G$ acts freely (on the
 right) with quotient $BG:=EG/G$. Then form the space
$$
EG\times_G X:=EG\times X/(e.g,x)\sim (e,g.x).
$$
In other words, $EG\times_G X$ represents the (topological) quotient
 stack $[X/G]$.

\begin{definition} The equivariant cohomology of $X$ with respect to
 $G$ is the ordinary singular cohomology of  $EG\times_G X$:
$$
H^i_G(X)= H^i(EG\times_G X).
$$
\end{definition}
For the special case when $X$ is a point, we have
$$
H^i_G(\{\point\})= H^i(BG).
$$
For any $X$, the map $X\rightarrow \{\point\}$ induces a pullback map $ H^i(BG)\rightarrow
 H^i_G(X)$. Hence the equivariant cohomology of $X$ has the structure
 of a $ H^i(BG)$-algebra, at least when $ H^i(BG)=0$ for odd $i$.

\subsection{Equivariant Chow groups $CH^i_G(X)$ of
 $X$}\cite{EdidinGraham}\label{equivchow}

As in the previous subsection, let $X$ be a smooth variety of dimension
 $n$, equipped with a left $G$-action. Here $G$ is an affine algebraic
 group of dimension $g$. Choose an $l$-dimensional representation $V$ of
 $G$ such that
$V$ has an open subset $U$ on which $G$ acts freely and whose
 complement has codimension more than $n-i$. The diagonal action on $X\times U$
 is also free, so there is a quotient in the category of algebraic
 spaces.
Denote this quotient by $X_G:=(X\times U)/G$.

\begin{definition}
The $i$-th equivariant Chow group $CH^G_i(X)$ is the usual Chow group \linebreak $CH_{i+l-g}(X_G)$. The codimension $i$ equivariant Chow group $CH^i_G(X)$ is the usual codimension $i$ Chow group $CH^i(X_G)$.  
\end{definition}

Note that if $X$ has pure dimension $n$ then
\begin{eqnarray*}
CH^i_G(X)&=& CH^i(X_G) \\
&=&CH_{n+l-g-i}(X_G) \\
&=&CH^G_{n-i}(X).\\
\end{eqnarray*}

\begin{proposition}
The equivariant Chow group $CH^G_i(X)$ is independent of the
 representation $V$, as long as $V-U$ has codimension more than $n-i$.
\end{proposition}
\begin{proof}
See \cite[Definition-Proposition 1]{EdidinGraham}.
\end{proof}

\begin{lemma}\label{equivclass}
If $Y\subset X$ is an $m$-dimensional subvariety which is invariant
 under the $G$-action, then it
 has a 
$G$-equivariant fundamental class
$[Y]_G \in CH^G_m(X)$. 
\end{lemma}
\begin{proof}
Indeed, we can consider the product $(Y\times
 U)\subset X\times U$, where $U$ is as above and the corresponding
 quotient $(Y\times U)/G$ canonically embeds into $X_G$. The fundamental
 class of $(Y\times U)/G$ defines the class $[Y]_G\in CH^G_m(X)$. More
 generally, if $V$ is an $l$-dimensional representation of $G$ and
 $S\subset X\times V$ is an $(m+l)$-dimensional subvariety which is invariant
 under the $G$-action, then the quotient $(S\cap (X\times U))/G \subset
 (X\times U)/G$ defines the $G$-equivariant fundamental class $[S]_G\in
 CH^G_m(X)$ of $S$.
\end{proof}

\begin{proposition}
If $\alpha\in CH^G_m(X)$ then there exists a representation $V$ such
 that $\alpha= \sum a_i[S_i]_G$, for some $G$-invariant subvarieties
 $S_i$
of $X\times V$.
\end{proposition} 
\begin{proof}
See \cite[Proposition 1]{EdidinGraham}.
\end{proof}

\subsection{Functoriality properties}

Suppose $f:X\rightarrow Y$ is a $G$-equivariant morphism.
Let $\mathcal{S}$ be one of the following properties of schemes or algebraic
 spaces: proper, flat, smooth, regular embedding or l.c.i.

\begin{proposition}
If $f:X\rightarrow Y$ has property $\mathcal{S}$, then the induced map
$f_G:X_G\rightarrow Y_G$ also has property $\mathcal{S}$.
\end{proposition}
\begin{proof}
See \cite[Proposition 2]{EdidinGraham}.
\end{proof}

\begin{proposition}\label{equivfunct}
Equivariant Chow groups (resp. equivariant cohomology) have the same functoriality as ordinary Chow
 groups (resp. cohomology) for equivariant morphisms with property $\mathcal{S}$.
\end{proposition}
\begin{proof}
See \cite[Proposition 3]{EdidinGraham}.
\end{proof}

If $X$ and $Y$ have $G$-actions then there are exterior products
$$
CH^G_i(X)\otimes CH^G_j(Y)\rightarrow  CH^G_{i+j}(X\times Y).
$$

In particular, if $X$ is smooth then there is an intersection product
on the equivariant Chow groups which makes $\oplus_j CH^G_j(X)$ 
into a graded ring.

\subsection{Cycle class maps} \cite[\S 2.8]{EdidinGraham}

Suppose $X$ is a complex algebraic variety and $G$ is a complex
 algebraic group. The equivariant Borel-Moore homology $H^G_{BM,i}(X)$ is the
 Borel-Moore homology  $H_{BM,i}(X_G)$, for $X_G=X\times_G U$. This is
 independent of the representation as long as $V-U$ has sufficiently large
 codimension. This gives a cycle class map,
$$
cl_i:CH^G_i(X)\rightarrow H_{BM,2i}^G(X,\mathbb{Z}),
$$
compatible with usual operations on equivariant Chow groups.
Suppose $X$ is smooth of dimension $d$ then $X_G$ is also smooth. In
 this case the Borel-Moore cohomology $H_{BM,2i}^G(X,\mathbb{Z})$ is dual to 
$H^{2d-i}(X_G)= H^{2d-i}(X\times_G U)$.

This gives cycle class maps
\begin{equation*}\label{cyclemap}
cl^i:CH^i_G(X)\rightarrow H^{2i}_G(X,\mathbb{Z}).
\end{equation*}

There are also maps from the equivariant groups to the usual groups:
\begin{equation}\label{liftcoh}
H^i_G(X,\mathbb{Z})\rightarrow H^i(X,\mathbb{Z})
\end{equation}
and
\begin{equation}\label{liftchow}
CH^i_G(X)\rightarrow CH^i(X).
\end{equation}

\begin{proposition}\label{Picequiv}
If $X$ is a smooth variety with a $G$-action, then the map
$$
\Pic^G(X)\rightarrow CH^1_G(X),\,\,L\mapsto c_1(L)
$$
is an isomorphism.
\end{proposition}
\begin{proof}
See \cite[Corollary 1]{EdidinGraham}.
\end{proof}

\subsection{Equivariant Chow groups and moduli stacks: \cite{Edidin}} \label{singmoduli}

Suppose $X$ is a complex variety and $G$ is an algebraic group acting on $X$. Let $\mathcal{X}$ denote the quotient stack $[X/G]$.
We refer to \cite{Edidin}, for a discussion on quotient stacks and the cohomology of the quotient stacks, with integral coefficients. We have:

\begin{theorem}\label{stackequivthm}
1) There is an equality of cohomology rings :
$$
H^*(\mathcal{X})= H^*_G(X).
$$
2) There is an equality of Chow rings :
$$
CH^*(\mathcal{X})= CH^*_G(X).
$$
\end{theorem}
\begin{proof}
See \cite[Theorem 3.16]{Edidin} and \cite[Proposition 19]{EdidinGraham} or \cite[Proposition 3.26]{Edidin}.
\end{proof}

\begin{proposition}\label{stackequiv}
If $X$ is smooth of dimension $n$ then there is an isomorphism
\begin{equation*}
CH^k(\mathcal{X}) \rightarrow  CH^G_{n-k}(X)
\end{equation*}
where $\mathcal{X}=[X/G]$. Moreover the ring structure on $CH^*(\mathcal{X})$ is given by composition of operations, and is compatible with the ring structure on $CH^G_*(X)$,
given by equivariant intersection product.
\end{proposition}
\begin{proof}
See \cite[Proposition 3.28]{Edidin}.
\end{proof}

Furthermore, if  $X$ is a smooth complex variety then there is a degree doubling cycle class map:
\begin{equation}\label{cyclemapstack}
cl: CH^*(\mathcal{X})\rightarrow H^*(\mathcal{X})
\end{equation}
having the same formal properties as the cycle class map on smooth complex varieties.

We will utilise the isomorphism in above proposition, to define the cohomology classes and relations amongst the Brill-Noether loci on the  moduli stack $\mathcal{U}_C(2, 2g-2)$.

\section{Tautological algebra generated by the Brill-Noether loci on $J_d(C)$}

In this section, we investigate the cohomology algebra generated by the Brill-Noether subvarieties of $J(C)$ and $J_d(C)$.  This problem is motivated by the classical Poincar\'e formula on $J(C)$.

\subsection{\textbf{Brill-Noether loci on $J(C)$}}
Let us fix a point $P\in C$.  Consider the classical Abel-Jacobi map $u:S^{d}(C) \longrightarrow J(C)$, where $u= \otimes \mathcal{O}(-dP)\circ \phi_d$ and $ \phi_d:S^d(C)  \rightarrow J_d(C)$, $\otimes \mathcal{O}(-dP):J_d(C)\rightarrow J(C)$ defined as follows. 
\begin{equation*}
\label{E100}
\xymatrix{ S^d(C)\ar[r]^{\phi_d} & J_d(C)
	\ar[r]^{\otimes \mathcal{O}(-dP)} &
	J(C)\\
	x_1+x_2+ \cdots +x_d\ar@{|->}[r] & \mathcal{O}(x_1+x_2+ \cdots +x_d)\ar@{|->}[r] & \mathcal{O}(x_1+x_2+ \cdots +x_d-dP).
}
\end{equation*}
\if
Therefore we have the following commutative diagram.
\begin{equation*}
\label{E200}
\xymatrix{ S^g(C)\ar[rr]^{\phi_g} && J_g(C)
	\ar[rr]^{\otimes \mathcal{O}(-gP)}\ar[d]_{\simeq}^{\otimes
		\mathcal{O}(-P)} &&
	J(C)\ar@{=}[d]\\ S^{g-1}(C)\ar[u]^{\text{inclusion}}\ar[rr]^{\phi_{g-1}}
	&& J_{g-1}(C)\ar[rr]^{\otimes \mathcal{O}(-(g-1)P)}
	\ar[d]_{\simeq}^{\otimes \mathcal{O}(-P)} & & J(C)\ar@{=}[d]
	\\ S^{g-2}(C)\ar[u]^{\text{inclusion}}
	\ar[rr]^{\phi_{g-2}} &&
	J_{g-2}(C)\ar[rr]^{\otimes\mathcal{O}(-(g-2)P)}
	\ar[d]_{\simeq}^{\otimes \mathcal{O}(-P)} & & J(C)\ar@{=}[d]
	\\ \vdots\ar[u]^{\text{inclusion}} && \vdots \ar[d]_{\simeq} &&
	\vdots\\ C \ar[u]^{\text{inclusion}}\ar[rr]^{\phi_1} && J_1(C)
	\ar[rr]^{\otimes (-P)} && J(C) \ar@{=}[u] }
\end{equation*}
The maps $\phi_d:S^d(C)  \rightarrow J_d(C)$'s for all $d$, $1\leq d\leq g$ are birational morphisms.  
\fi

Now define $W^0_{d}$, for all $d$, $1\leq d\leq g$, called the Brill-Noether subvarieties of J(C), as follows:
$$W^0_{d}:= u(S^{d}(C)).$$
Let $\Theta:=u(S^{g-1}(C))$.  The classical Poincar\'e relations determine the relations between the cohomological classes of $W^0_i$ on $J(C)$:
\begin{equation*}
[W^0_i]=\frac{1}{(g-i)!}[\Theta]^{g-i} \,\in\, H^*(J(C),\mathbb{Q}).
\end{equation*}
See {\cite[Chapter 1, \S 5,  p.~25]{ACGH}}.

\subsection{\textbf{Brill-Noether loci in $J_d(C)$}} 
For a fixed $d$, we recall the  Brill-Noether locus $W^r_d$, which are defined to be certain natural closed subschemes of $J_d(C)$ and discuss some of  its properties relevant to us.

\begin{definition}\label{set theoretic definition of Brill Noether loci}
	As a set, for $r\geq 0$, we define
	$$
	W^r_d:=\{L\in J_d(C): h^0(L)\geq r+1\}\subseteq J_d(C).
	$$ 
\end{definition}
It is clear from semicontinuity theorem (cf. \cite[Theorem 12.8 ; p. 288]{H}) that $W^r_d$ is closed. In fact, $W^r_d$ has a natural scheme structure as determinantal locus (cf. \cite[\S 4, Chapter \RomanNumeralCaps{2}; p. 83]{ACGH}) of certain morphisms of vector bundles over $J_d(C)$. 
We define these morphisms as follows:

Let us fix a Poincar\'e bundle $\mathcal{L}$ over $C\times J_d(C)$. Let $E$ be an effective divisor on $C$ with 
$$\text{deg}(E)= m \geq 2g-d-1 .$$
Let $\Gamma:= E\times J_d(C)$.
Then, over $C\times J_d(C)$ 
we have the exact sequence:

\begin{equation}\label{sequence over Ctimes J(C)}
0 \to \mathcal{L} \to \mathcal{L}(\Gamma) \to \mathcal{L}(\Gamma)|_{\Gamma} \to 0 .
\end{equation}

Let $v$ be the projection from $C\times J_d(C) \to J_d(C)$.
Now, applying the functor $v_*$ to the morphism 
$\mathcal{L}(\Gamma) \to \mathcal{L}(\Gamma)|_{\Gamma}$ as in (\ref{sequence over Ctimes J(C)}), we get a morphism 
$$\gamma:= v_*(\mathcal{L}(\Gamma))\to v_*(\mathcal{L}(\Gamma)\vert_{\Gamma}) .$$

Note that, by the choice of the degree of $E$ and Grauert's theorem (cf. \cite[ Corollary. 12.9, p. 288]{H}), we get that both $v_*(\mathcal{L}(\Gamma))$ and $v_*(\mathcal{L}(\Gamma)|_{\Gamma})$ are vector bundles of rank $d+m-g+1$ and $m$ respectively. 

\begin{definition}\label{scheme theoretic definition of Brill Noether loci}
	The Brill-Noether locus $W^r_d$ is defined to be the $(m+d-g-r)$-th determinantal locus associated to the morphism $\gamma$.
\end{definition}

To see that Definition \ref{scheme theoretic definition of Brill Noether loci} indeed agrees with Definition \ref{set theoretic definition of Brill Noether loci}, in the sense that the set theoretic support of \ref{scheme theoretic definition of Brill Noether loci} is exactly \ref{set theoretic definition of Brill Noether loci}, we refer to \cite[Lemma 3.1; p. 178]{ACGH}.

From general properties of determinantal loci, we have the following lemma:

\begin{lemma}\cite[Lemma 3.3; p. 181]{ACGH}\label{dimension of BN loci}
	Suppose $r\geq d-g$. Then every component of $W^r_d$ has dimension greater or equal to the Brill-Noether number 
	\begin{equation*}
	\rho:=g-(r+1)(g-d+r).
	\end{equation*}
\end{lemma}

\begin{remark}
	Note that if $r\leq d-g-1$, then by Riemann-Roch theorem $W^r_d=J_d(C)$. So, from here onwards, we will assume that $r\geq d-g$. 
\end{remark}

In general, the above inequality can be strict (cf. \cite[Theorem 5.1; p. 191]{ACGH}). Even, in the case when equality holds, $W^r_d$ can have more than one components,  (cf. \cite[Chapter \RomanNumeralCaps{5}; p. 208]{ACGH}).

We recall the following theorem due to Griffith and Harris.

\begin{theorem}\label{dimension of BN loci for general curve}\
a)	For any smooth projective curve $C$ of genus $g$ 
$$
\dim(W^r_d)\geq \rho .
$$
b) For a general curve $C$
$$
\dim(W^r_d)= \rho.
$$
Furthermore,
$$
	[W^r_d]=\prod\limits^r_{\alpha=0}\dfrac{\alpha!}{(g-d+r+\alpha)!}.[\Theta_d]^{g-\rho}.
	$$
\end{theorem}
\begin{proof}
	See \cite[Main Theorem; p. 235]{Gr-Ha2}.
\end{proof}
When $\rho=0$ the above formula is called the Castelnuovo formula.  Regarding the irreducibility, we have:

\begin{theorem} \label{irreducibility of BN loci for general curve}
	If $C$ is general and $\rho > 0$, then $W^r_d$ is irreducible.
\end{theorem}
\begin{proof}
	See \cite[Corollary 2.4; p. 280]{Fu-La}.
\end{proof}
\if
Now, recall that in the case when $C$ is general, by Theorem \ref{T4} we have that the N\'eron Severi group of $J_d(C)$ is generated by a translate of the $\Theta$ divisor in $J(C)$. We denote this class as $\theta_d$. In particular, this implies that the class of $W^r_d$ can be written in terms of powers of $\theta_d$. 
\fi

\subsection{Tautological algebra generated by the Brill-Noether loci in $J(\widetilde{C})$}\label{JSpec}

Let $\pi:\widetilde{C}\to C$ be a spectral curve which was defined in \S \ref{spectral}.
In this section we investigate the  subalgebra of $H^{\ast}(J(\widetilde{C}),\mathbb{Q})$ generated by the Brill-Noether loci on $J(\widetilde{C})$. Towards this, we consider the case when we have a ramified double cover $\pi:\widetilde{C}\to C$. \\
Let $\mathcal{R}_g^r$ denote the moduli space of ramified two sheeted  covering of  a connected smooth projective curves of genus $g$ with fixed ramification $r$.  Then we have the following theorem.
\begin{theorem}\label{T6}
	The N\'eron-Severi group of the Jacobian of a general element of $\mathcal{R}_g^r$ is generated by two elements; the two elements are obtained from the decomposition (up to isogeny) of the Jacobian of a covering curve (see \S \ref{Prym}).  Furthermore, the N\'eron-Severi group generates the algebra of Hodge cycles (of positive degree) on the Jacobian of the general double cover.
\end{theorem}
\begin{proof}
	See {\cite[Corollary 5.3; p.~634]{B}}.
\end{proof}
Note that even if $C$ is general, $\widetilde{C}$ may not be general. However, in our situation, we will check that the above theorem still holds.

\begin{theorem}\label{Bis}
	The cohomology class of a Brill-Noether locus on the Jacobian  
	$J(\widetilde{C})$ of a general $2$-sheeted spectral curve $\pi:\widetilde{C}\rightarrow C$ can be expressed  
	as a sum of powers of divisor classes. In particular the tautological algebra is generated by the divisor classes.
\end{theorem}
\begin{proof}
We only need to check that Theorem \ref{T6} can be applied to the  Jacobian of a general spectral curve. Fix a degree $d>0$.
Denote  $\mathcal{S}_{g,s}$ the moduli space of tuples 
$$
\{(C, L, s=(s_0,s_1)\},
$$
 where $C$ is a curve of genus $g$, $L$ is a line bundle on $C$ of degree $d$, and $s_0\in H^0(C,L),\, s_1\in H^0(C, L^2)$.
This moduli space can be interpreted as the moduli space of spectral curves, as in \S \ref{sec:Spectral curves}.  There is a dominant rational map (on the component $\mathcal{S}^{0}_{g,s}$ where $(s_0=0)$)
$$
\mathcal{S}^{0}_{g,s}\rightarrow \mathcal{R}_g^r\rightarrow \mathcal{U}_g.
$$
Here $r$ is the ramification type corresponding to a general section $s$ equivalently the zeroes of the equation \eqref{SC}  (cf. \cite{BCV} for a similar moduli space).  
The maps are given by
$$
(C, L, s) \mapsto (C, L, B(s)) \mapsto C,
$$
where $B$ is the branch divisor  of the spectral curve $\tilde{C_s}\rightarrow C$, such that $L^2=\mathcal{O}(B)$.
Since $J(\tilde{C_s})$ depends only the ramification type $B$ and $L$,  Theorem \ref{T6} can be applied to the Jacobian of a  general spectral curve.

\end{proof}


\section{Brill-Noether loci on ${U}_C(n,d)$}\label{sec:Brill-Noether loci on U}

To define the Brill-Noether loci for ${U}_C(n,d)$, we start with a more general set up. Let $S$ be an algebraic scheme over $\mathbb{C}$. Let $\mathcal{E}$ be a vector bundle  over $C\times S$ such that for all  $s\in S$, $\mathcal{E}_s:=\mathcal{E}|_{C\times s}$ is a vector bundle of rank $n$ and degree $d$ over $C$.

Just as in Definition \ref{set theoretic definition of Brill Noether loci}, we have the following definition of the Brill-Noether locus as a closed set.
\begin{definition}\label{set theoretic definition of Brill Noether loci for Moduli spaces}
We define the Brill-Noether locus $W^{r}_{S,\mathcal{E}}$ associated to a pair $(S,\mathcal{E})$:
$$
W^r_{S,\mathcal{E}}:=\{s\in S| h^0(C,\mathcal{E}_s)\geq r+1 \} .
$$
\end{definition}
By \cite[Lemma 1.7.6; p. 28]{Hu-Ln}, since the family $\mathcal{E}$ is a bounded family, we can choose a divisor $D$ in $C$ of sufficiently high degree such that $h^1(C,\mathcal{E}_s(D))=0$ for all $s\in S$. For notational convenience, we continue to denote the pullback of $D$ to $C\times S$ by $D$. Then, over $C\times S$ we have the exact sequence:
$$
0 \to \mathcal{E} \to \mathcal{E}(D) \to \mathcal{E}(D)|_{D}\to 0 .
$$

Let $v:C\times S\to S$ be the projection.

Then, we have the morphism 
$$
f:v_*(\mathcal{E}(D))\to v_*(\mathcal{E}(D)|_{D}) .
$$

Now for any $s\in S$ we have $h^1(C,\mathcal{E}(D)_s)=h^1(C,\mathcal{E}_s(D))=0$. By Riemann-Roch theorem we get 
\begin{align*}
h^0(C,\mathcal{E}(D)_s)= & ~d+n~ \text{deg}(D)+n(1-g), \\
h^0(C,(\mathcal{E}(D)|_D)_s)= & ~n~\text{deg}(D).
\end{align*}
Hence, by \cite[Theorem 12.11; p. 290]{H}, we get that both $v_*(\mathcal{E}(D))$ and $v_{*}(\mathcal{E}(D)|_{D})$ are vector bundles and for any $s\in S$, we have isomorphisms:
\begin{equation}\label{isomorphisms}
\begin{split}
v_*(\mathcal{E}(D))|_s \xrightarrow{\cong} & H^0(C,\mathcal{E}|_{C\times s}(D)),\\
 v_{*}(\mathcal{E}(D)|_{D})|_s \xrightarrow{\cong} & H^0(C,\mathcal{E}|_{C\times s}(D)|_{D}).
\end{split}
\end{equation}
Using Riemann-Roch theorem, we get that 
\begin{align*}
\text{rank}(v_*(\mathcal{E}(D)))             = &~ d+n~\text{deg}(D)+n(1-g),  \\
\text{rank}(v_{*}(\mathcal{E}(D)|_{D})) = & ~n ~\text{deg}(D) .
\end{align*} 
\begin{definition}\label{scheme theoretic definition of Brill Noether loci for Moduli spaces}
We define $W^r_{S,\mathcal{E}}$ to be the $(d+n~ \text{deg}(D)+n(1-g) -(r+1))$-th determinantal locus associated to the morphism $f$. 
\end{definition}
\begin{remark}
To see that the set-theoretic support of \ref{scheme theoretic definition of Brill Noether loci for Moduli spaces} is indeed \ref{set theoretic definition of Brill Noether loci for Moduli spaces}, note that we have the following commutative diagram:

\[
\begin{tikzcd}
 &  & v_*(\mathcal{E}(D))|_s \arrow[r,"f|_s"] \arrow[d,"\cong"] &  v_{*}(\mathcal{E}(D)|_{D})|_s \arrow[d,"\cong"]\\
0 \arrow[r] & H^0(C,\mathcal{E}_s) \arrow[r] & H^0(C,\mathcal{E}_s(D)) \arrow[r] & H^0(C,\mathcal{E}_s(D)|_{D})
\end{tikzcd}
\]

Hence, 
$$
 \text{rank}(f|_s)\leq d+n~ \text{deg}(D)+n(1-g) -(r+1) \Longleftrightarrow h^0(C,\mathcal{E}_s)\geq r+1.
 $$

From this, it follows that Definition \ref{scheme theoretic definition of Brill Noether loci for Moduli spaces} agrees with Definition \ref{set theoretic definition of Brill Noether loci for Moduli spaces}.
\end{remark}
\begin{lemma}\label{dimension of Brill-Noether loci of moduli spaces}
If $W^r_{S,\mathcal{E}}\neq \emptyset$, then, codimension of each component of $W^r_{S,\mathcal{E}}\leq (r+1)(r+1-d+n(g-1))$.
\end{lemma}
\begin{proof}
This follows from \cite[\S 4, Chapter \RomanNumeralCaps{2}; p. 83]{ACGH}.
\end{proof}
\begin{lemma}\label{functoriality of Brill-Noether loci}
Let $S_1, S_2$ be two algebraic schemes over $\mathbb{C}$ and let $\mathcal{E}$ be a bundle on $C\times S_2$ such that for all $s\in S_2,~\mathcal{E}_s$ is a vector bundle of rank $n$ and degree $d$. If $g:S_1\to S_2$ be a morphism, then
\begin{equation*}
g^{-1}(W^r_{S_2,\mathcal{E}})=W^r_{S_1,(id_C \times g)^* \mathcal{E}}.
\end{equation*} 
\end{lemma}

\begin{proof}
Let $v_1:C\times S_1\to S_1$ and $v_2:C\times S_2\to S_2$ be the projections.
Let $G:=id_C\times g:C\times S_1\to C\times S_2$.
Then we have the following commutative diagram:
\[
\begin{tikzcd}
C\times S_1 \arrow[r,"G"] \arrow[d,"v_1"] & C\times S_2 \arrow[d,"v_2"] \\
        S_1 \arrow[r,"g"]                            &         S_2
\end{tikzcd}
\]
This induces the following commutative diagram:
\[
\begin{tikzcd}
g^*(v_2)_*(\mathcal{E}(D)) \arrow[r] \arrow[d] & g^*(v_2)_*(\mathcal{E}(D)|_{D}) \arrow[d] \\
(v_1)_*G^*(\mathcal{E}(D)) \arrow[r] & (v_1)_*G^*(\mathcal{E}(D)|_{D}) 
\end{tikzcd}
\]
By (\ref{isomorphisms}), we get that the vertical arrows in the above diagram are isomorphisms.  Now, the lemma follows from general properties of determinantal loci. 
\end{proof}

Now suppose $\widetilde{C}$ be a smooth projective curve of genus $\widetilde{g}$ and $\pi:\widetilde{C} \to C$ be a finite morphism. Let $\mathcal{E}$ be a vector bundle over $\widetilde{C}\times S$ such that $\mathcal{E}_s$ is of rank $n$ and degree $d$ for all $s\in S$. Since the map $\pi \times id:\widetilde{C} \times S \to C\times S$ is a finite flat morphism, we get that $(\pi \times id)_*\mathcal{E}$ is a vector bundle over $C\times S$ and in fact,
$$
((\pi \times id)_*\mathcal{E})_s=\pi_*(\mathcal{E}_s) ~\text{for all}~ s\in S.
$$
We will denote this bundle $(\pi \times id)_*\mathcal{E}$ by $\mathcal{E}'$. Note that rank of $\mathcal{E}'$ is 
$$
n':=n~\text{deg}(\pi),\, \text{for all}~ s\in S.
$$  
Let $d':={\rm deg}(\mathcal E'_s)$.
Since $\mathcal E'_s=\pi_*(\mathcal E_s)$, by Riemann-Roch we have 
$$
d+n(1-g)=\chi(\widetilde{C},\mathcal E_s)=\chi(C,\mathcal E'_s)=d'+n'(1-g).
$$
Hence we have 
$$
d'=d+n(1-\widetilde{g})-n~\text{deg}(\pi)(1-g).
$$  
Then we have the following lemma:
\begin{lemma}\label{Brill Noether loci of pushforward of a bundle}
$W^r_{S,\mathcal{E}}=W^r_{S,\mathcal{E}'}\,.$
\end{lemma}
\begin{proof}
We have the commutative diagram:
\[
\begin{tikzcd}
\widetilde{C} \times S \ar[r,"\pi \times id"] \arrow[dr,"\widetilde{v}"'] & C\times S \ar[d,"v"] \\
                                             & S
\end{tikzcd}
\]
Fix $D$ a divisor on $C$ such that $h^1(\mathcal{E}'_s(D))=0~\text{for all}~ s\in S$. Then $$h^1(\widetilde{C}, \mathcal{E}_s(\pi^{*}D))=h^1(C,\pi_*(\mathcal{E}_s(\pi^*D)))=h^1(C,\mathcal{E}'_s(D))=0$$
Therefore we can use the divisor $\pi^*D$ for the construction of $W^r_{S,\mathcal{E}}$.

Let us denote the morphism 
$$
f:\mathcal{E}(\pi^*D)\to (\mathcal{E}(\pi^*D))|_{\pi^*D}.
$$ 
Then $W^r_{S,\mathcal{E}}$ is defined to be the 
$
(d+n~\text{deg}(\pi^*D)+n(1-\widetilde{g}) -(r+1))
$-th determinantal locus of the morphism $\widetilde{v}_*f$. Now $\widetilde{v}=v\circ (\pi \times id)$.
  It follows from projection formula that $(\pi \times id)_*f$ is nothing but the morphism 
  $$
  \mathcal{E}'(D)\to \mathcal{E}'(D)|_D
  $$
   and therefore, $W^r_{S,\mathcal{E}'}$ is the $(d'+n'~\text{deg}(D)+n'(1-g)-(r+1))$-th determinantal locus of $v_*(\pi \times id)_*f=\widetilde{v}_*f$. It can be checked easily that 
   $$
   d'+n'~\text{deg}(D)+n'(1-g)-(r+1)=d+n~\text{deg}(\pi^*D)+n(1-\widetilde{g})-(r+1).
   $$
\end{proof}
Next, we will define the Brill-Noether Loci for ${U}_C(n,d)$. Note that if $(n,d)=1$, we have a universal bundle over $C\times {U}_C(n,d)$ (cf. \cite[Corollary 4.6.6; p. 119]{Hu-Ln}) and hence, we can apply the previous construction to get the notion of the Brill-Noether loci in this case. However, in general we don't have a universal bundle. 

Recall that ${U}_C(n,d)$ was constructed as a good quotient of certain Quot Schemes (cf. \cite[ \S 4.3; p. 88]{Hu-Ln}). We recall the definition of this Quot Scheme.
Fix a line bundle ${O}(1)$ of degree $1$ over $C$. 
Choose an $m \gg 0$ such that any semistable vector bundle $E$ over $C$ of rank $n$ and degree $d$ is $m$-regular. 

In particular, we have that 
\begin{equation}\label{equation for N}
\begin{split}
(1)\; &  h^1(C,E(m))=0\\
(2)\; & h^0(C,E(m))=d+mn+n(1-g)=:N\\
(3)\; & \text{The natural map}\; H^0(C,E(m))\otimes \mathcal{O}\to E(m) \;\text{is surjective}.
\end{split}
\end{equation}
Now, define $\mathcal{Q}$ to be the Quot Scheme parametrizing quotients of $\mathcal{O}^N$ of rank $n$ and degree $d+mn$. Let 
$$
\mathcal{O}^N_{C\times \mathcal{Q}}\to \mathcal{F}
$$ 
be the universal quotient.

Note that the group scheme $\GL(N)$ acts on $\mathcal{Q}$ in the following manner:

Let $T$ be an algebraic scheme over $\mathbb{C}$.

Let $g\in \GL(N)(T)$ be an automorphism $\mathcal{O}^N_{C\times T}\xrightarrow{g} \mathcal{O}^N_{C\times T}$. Let $[\mathcal{O}^N_{C\times T}\to F_T]\in \mathcal{Q}(T)$.

Then, define 
$$
g.[\mathcal{O}^N_{C\times T}\to F_T] := [\mathcal{O}^N_{C\times T}\xrightarrow{g} \mathcal{O}^N_{C\times T}\to F_T].
$$ 

It is clear that this action in fact factors through an action of the group scheme $\PGL(N)$.

Let $\mathcal{R}\subseteq \mathcal{Q}$ be the open subset such that for all $x\in \mathcal{R}$, 
$\mathcal{F}|_{C\times x}$ is a semistable bundle and $H^0(C,\mathcal{O}^N) \to H^0(C,\mathcal{F}|_{C\times x})$ is an isomorphism.  It is immediate that $\mathcal{R}$ is $\PGL(N)$-equivariant. Then, we define 

$$
{U}_C(n,d):=\mathcal{R}\sslash_{\PGL(N)}\,.
$$
and we have the quotient map

\begin{equation}\label{classifyingmap}
\mu:\mathcal{R}\to {U}_C(n,d) .
\end{equation}

Let us denote $\mathcal{F}|_{C\times \mathcal{R}}$ by $\mathcal{F}'$.
By Definition \ref{scheme theoretic definition of Brill Noether loci for Moduli spaces} we have the closed subscheme $W^r_{\mathcal{R},\mathcal{F}'(-m)}\subseteq \mathcal{R}$.

\begin{lemma}\label{GLn-invariantlemma}
 $W^r_{\mathcal{R},\mathcal{F}'(-m)}$ is $\GL(N)$-equivariant (consequently $\PGL(N)$-equivariant as well), and compatible with the $\GL(N)$-action on $\mathcal{R}$.
\end{lemma}
\begin{proof} 
   Let $q:T\to W^r_{\mathcal{R},\mathcal{F}'(-m)}$ be a $T$-valued point of $W^r_{\mathcal{R},\mathcal{F}'(-m)}$.
Let $\mathcal{O}^N_{C\times T}\to F_T$ be the pullback of the universal quotient under  $q$. By Lemma \ref{functoriality of Brill-Noether loci}, we get that 
$W^r_{T,F_T(-m)}=q^{-1}(W^r_{\mathcal{R},\mathcal{F}'(-m)})=T$.
Let $g\in \GL(N)(T)$. By definition, the quotient corresponding to 
$g.q:T\to \mathcal{R}$
is given by 
$$
\mathcal{O}^N_{C\times T}\xrightarrow{g} \mathcal{O}^N_{C\times T} \to F_T .
$$
Again, by Lemma \ref{functoriality of Brill-Noether loci} we have
$$
(g.q)^{-1}(W^r_{\mathcal{R},\mathcal{F}'(-m)})=W^r_{T,F_{T}(-m)}=T .
$$
In other words, we get that $g.q:T\to \mathcal{R}$ factors through $W^r_{\mathcal{R},\mathcal{F}'(-m)}$. Hence the closed subscheme $W^r_{\mathcal{R},\mathcal{F}'(-m)}$ is $\GL(N)$-equivariant.
\end{proof}

\if
Since,
 $$
 \pi:\mathcal{R}\to \mathcal{R}//\PGL(N)
 $$
 is a good quotient, $B^r_{\mathcal{R},\mathcal{F}'(-m)}$ descends to a closed subscheme.

\begin{definition}
We define $B^r_{{U}_{C}(n,d)}(\forall)\subseteq {U}_{C}(n,d)$ to be the closed subscheme whose pullback to $\mathcal{R}$ is $B^r_{\mathcal{R},\mathcal{F}'(-m)}$.
\end{definition}
Note that as sets,
$$
|B^r_{\mathcal{U}_{C}(n,d)}(\forall)|=\{ e \in U_{C}(n,d)~|~ H^0(C,E)\geq r+1 ~ \forall E\in e\}.
$$
\fi
\begin{definition}\label{Brill-Noether loci over semistable locus}
We define the Brill-Noether locus $\widetilde{W^r_{n,d}}(C)$ to be the scheme theoretic image of $W^r_{\mathcal{R},\mathcal{F}'(-m)}$ under the morphism $\mu$. In other words,
$$
\widetilde{W^r_{n,d}}(C)= \mu (W^r_{\mathcal{R},\mathcal{F}'(-m)}).
$$

\end{definition}
\begin{notation}
We will denote $\widetilde{W^r_{n,d}}(C)$ by $\widetilde{W^r_{n,d}}$ when there is no chance of confusion.
\end{notation}

\begin{remark}\label{GLn-invariant}
Note that since the morphism 
$$
\mu:\mathcal{R}\to {U}_C(n,d)
$$
is a good quotient and $W^r_{\mathcal{R},\mathcal{F}'(-m)}$ is $\PGL(N)$-equivariant, we get that $\mu(W^r_{\mathcal{R},\mathcal{F}'(-m)})$ is a closed subset of ${U}_{C}(n,d)$. Hence, as sets
$\widetilde{W^r_{n,d}}=\mu(W^r_{\mathcal{R},\mathcal{F}'(-m)})$.  That is to say, denoting the $S$-equivalence class of a semistable bundle $E$ over $C$ by $e$, we get
\begin{equation} \label{B-L in semistable loci}
\widetilde{W^r_{n,d}}=\{ e \in \mathcal{U}_{C}(n,d)~|~\exists\; E\in e \text{ such that } h^0(C,E)\geq r+1 \} .
\end{equation}
\end{remark}

Let us denote by ${U}^s_C(n,d)$ the moduli space of stable bundles on $C$ of rank $n$ and degree $d$. Recall that ${U}^s_{C}(n,d)$ is an open subset of 
${U}_C(n,d)$.

\begin{definition}\label{Brill-Noether loci over stable locus}
We define the Brill-Noether locus $W^r_{n,d}$ of ${U}^s_{C}(n,d)$ to be the closed subscheme 
$$
W^r_{n,d}:=\widetilde{W^r_{n,d}} \cap {U}^s_{C}(n,d)\subset {U}^s_C(n,d).
$$
\end{definition}
\begin{remark}
Let $\mathcal{R}^s\subseteq \mathcal{R}$ be the set of all $x\in \mathcal{R}$ such that $\mathcal{F}'|_{C\times x}$ is stable.  Let $\mathcal{F}'':= \mathcal{F}'|_{C\times \mathcal{R}^s}$. Let $\mu_s:\mathcal{R}^s\to {U}^s_{C}(n,d)$ be the restriction of $\mu$ to $\mathcal{R}^s$. Then, $W^r_{n,d}$ is the scheme-theoretic image of $W^r_{\mathcal{R}^s,\mathcal{F}''(-m)}$ under the map $\mu_s$.
\end{remark}

Now, using the fact that $\mu_s:\mathcal{R}^s\to {U}^s_C(n,d)$ is a principal $\PGL(N)$-bundle (cf. \cite[Corollary 4.3.5; p. 91]{Hu-Ln}), and Lemma \ref{dimension of Brill-Noether loci of moduli spaces} we have the following lemma:

\begin{lemma}\label{dimension of brill noether loci for stable locus}
If $W^r_{n,d}\neq \emptyset$, then dimension of each component of  $W^r_{n,d}$ is at least 
$$
n^2(g-1)+1-(r+1)(r+1-d+n(g-1)).
$$
\end{lemma}
\begin{definition}\label{expected dimension of brill noether loci for stable locus}
We define 
$$
\rho^r_{n,d}:=n^2(g-1)+1-(r+1)(r+1-d+n(g-1))
$$
 to be the expected dimension of $W^r_{n,d}$.
\end{definition}
\begin{remark}
The above lemma is not true in the case of $\widetilde{W^r_{n,d}}$. It may have components whose dimensions are less than $\rho^r_{n,d}$. See \cite[\S 7]{BGN} for example.
\end{remark}
\begin{lemma}\label{functoriality of brill noether loci of moduli spaces of stable bundles}
Let $S$ be an algebraic scheme and $\mathcal{E}$ be a vector bundle over $C\times S$ such that for all $s\in S, \mathcal{E}_s$ is stable of rank $n$ and degree $d$. If $f:S\to {U}^s_C(n,d)$ is the induced map, then 
$$
f^{-1}(W^r_{n,d})=W^r_{S,\mathcal{E}}.
$$
\end{lemma}
\begin{proof}
First we show that the statement is true in the case when $S=\mathcal{R}^s$ and $\mathcal{E}=\mathcal{F}''(-m)$. As we saw earlier, $W^r_{\mathcal{R}^s,\mathcal{F}''(-m)}$ is a $\PGL(N)$-equivariant subscheme and since $\mathcal{R}^s\to {U}^s_C(n,d)$ is a principal $\PGL(N)$-bundle, $W^r_{\mathcal{R}^s,\mathcal{F}''(-m)}$ descends to a closed subscheme $Z$ in ${U}^s_C(n,d)$, i.e. 
$$
\mu^{-1}_s(Z)=W^r_{\mathcal{R}^s,\mathcal{F}''(-m)}.
$$
 Since $W^r_{n,d}=\mu_s(W^r_{\mathcal{R}^s,\mathcal{F}''(-m)})$, it is clear that $Z=W^r_{n,d}$. Hence 
 $$
 \mu^{-1}_s(W^r_{n,d})=W^r_{\mathcal{R},\mathcal{F}''(-m)}.
 $$
Now let $(S,\mathcal{E})$ be as in the hypothesis. Since $F''(-m)$ is a locally universal family, for any $x\in S$, there exists $U_x\subset \mathcal{R}$ which is open and a map $g:U_x\to \mathcal{R}^s$ such that 
$$
(id\times g)^*(\mathcal{F}''(-m))=\mathcal{E}|_{C\times U_x}.
$$
 By Lemma \ref{functoriality of Brill-Noether loci} we have 
 $$
 g^{-1}(W^r_{\mathcal{R},\mathcal{F}''(-m)})=W^r_{S,\mathcal{E}}\cap U_x.
 $$
  Since $\mu_s \circ g=f|_{U_x}$, we have 
  $$
  (f|_{U_x})^{-1}(W^r_{n,d})=W^r_{S,\mathcal{E}}\cap U_x.
  $$
   The lemma now follows from this.
\end{proof}

Our aim in this paper is to find a Poincar\'e type expression for the cohomology class, in the cohomology ring of the  moduli stack $\mathcal{U}_C(2,d)$ (resp. $\mathcal{SU}_C(2, \mathcal{L})$ for a line bundle $\mathcal{L}$ on $C$).

Let us now fix different notations of the Brill-Noether subvarieties in different spaces to avoid confusion.

 For a given scheme and for a given sheaf $\mathcal{E}$ over $C\times S$, we denote the Brill-Noether locus by $W^r_{S,\mathcal{E}}$ as in Definition \ref{set theoretic definition of Brill Noether loci for Moduli spaces} or in Definition \ref{scheme theoretic definition of Brill Noether loci for Moduli spaces}.  We also denote this by $W^r_{S}$ when the sheaf involved is clear from the context.  In $\mathcal{U}_C(n,d)$ the Brill-Noether locus is denoted by $\widetilde{W^r_{n,d}}$ as in (\ref{B-L in semistable loci}).  The same is denoted by $W^r_{n,d}$ in ${U}^s_C(n,d)$ as in Definition \ref{Brill-Noether loci over stable locus}.  Inside $J_d(C)$, that is inside ${U}_C(1,d)$, the Brill-Noether locus $W^r_{1,d}$ is denoted by $W^r_d$ as in Definition \ref{set theoretic definition of Brill Noether loci} or in Definition \ref{scheme theoretic definition of Brill Noether loci}.  Inside $J_d(\widetilde{C})$ the same is denoted by $W^r_{d}(\widetilde{C})$.

\section{Chow-Cohomology class of the Brill-Noether locus}\label{CohBN}

For generalities on algebraic stacks, and in particular on moduli stacks of vector bundles, we refer to \cite{Gomez}. Denote $G:=\GL(N)$.

Consider the semistable locus $\mathcal{R}$ of the Quot scheme, together with the quotient morphism \eqref{classifyingmap} to the GIT-quotient:

\begin{equation*}
\mu:\mathcal{R}\to {U}_C(n,d)=\mathcal{R}\sslash \PGL(n).
\end{equation*}

By \cite{DN}, $\mathcal{R}$ is a smooth variety. 
Recall that the Brill-Noether locus is defined as (cf. Definition \ref{Brill-Noether loci over semistable locus})
$$
\widetilde{W^r_{n,d}}= \mu (W^r_{\mathcal{R},\mathcal{F}'(-m)}).
$$

Also consider the map to the quotient stack (see \cite[Proposition 3.3,p.23]{Gomez}):
\begin{equation*}
\mu_{st}:\mathcal{R}\to \mathcal{U}_C(n,d)=[\mathcal{R}/\GL(n)].
\end{equation*}
We use the same notation for the Brill-Noether loci in the moduli stack and its coarse moduli scheme.

We first notice the following:
\begin{lemma}\label{BNequivclass}
The Brill-Noether loci give well-defined Chow-cohomology class in the equivariant Chow-cohomology of $\mathcal{R}$, and is compatible with the cycle class map on equivariant groups.
\end{lemma}
\begin{proof}
Using Lemma \ref{GLn-invariantlemma} and Lemma \ref{equivclass}, we know  that $W^r_{\mathcal{R},\mathcal{F}'(-m)}$ is equivariant for the $\GL(n)$-action.
Hence, it corresponds to a $\GL(n)$-equivariant cohomology class: 
$$
[\widetilde{W^r_{n,d}}] \,\in\, \oplus_i H^{2i}_G(\mathcal{R},\mathbb{Z}).
$$
Similarly, we obtain an equivariant Chow class:
$$
[\widetilde{W^r_{n,d}}] \,\in\, \oplus_i CH^{i}_G(\mathcal{R}).
$$

Since we do not know if the Brill-Noether locus is of pure dimension, we will use the cycle class map on the equivariant  Chow ring:
$$
CH^*_G(\mathcal{R})\rightarrow \oplus_i H^{2i}_G(\mathcal{R}, \mathbb{Z}).
$$
Via the cycle class map, the equivariant Chow class
$$
[\widetilde{W^r_{n,d}}] \in CH^*_G(\mathcal{R})=\oplus_i CH^i_G(\mathcal{R})
$$
 maps to the Brill-Noether cohomology class
$$
[\widetilde{W^r_{n,d}}] \,\in\,  \oplus_i H^{2i}_G(\mathcal{R},\mathbb{Z}).
$$
\end{proof}


\subsection{Brill Noether Chow-Cohomology class on the moduli stack $\mathcal{U}_C(r,2g-2)$}

Henceforth, we write $d=2g-2$. (Note that most of the discussions remain true for other degrees as well, except for certain results in the final section).

Since ${U}_C(2,2(g-1))$ is a singular variety, we will consider Chow groups $CH^*(\mathcal{U}_C(r,d))$ of the moduli stack, instead of the Chow groups of the coarse moduli scheme $U_C(2,2g-2)$:

\begin{lemma}
The Brill Noether loci give a well-defined Chow class
$$
[\widetilde{W^r_{n,d}}] \in CH^*(\mathcal{U}_C(n,d))
$$
and a cohomology class
$$
[\widetilde{W^r_{n,d}}] \in H^*(\mathcal{U}_C(n,d),\mathbb{Z}).
$$
The classes are compatible under the cycle class map, as in \eqref{cyclemapstack}.
\end{lemma}
\begin{proof}
Since $\mathcal{R}$ is a smooth variety, the lemma follows from the identification in Proposition \ref{stackequiv}, and above Lemma \ref{BNequivclass}.
\end{proof}

In particular, we have the following lemma.

\begin{lemma}\label{qchowcoh}
The Brill-Noether class $[W^r_{\mathcal{R}}]$ is non-zero if and only if $[\widetilde{W^r_{n,d}}]$ is non-zero, in Chow-cohomology (resp. in  cohomology ring).
\end{lemma}
\begin{proof}
Using \eqref{liftchow} and Theorem \ref{stackequivthm}   (similarly for cohomology: use \eqref{liftcoh}), we have:
$$
CH^*(\mathcal{U}_C(r,d))=CH^*_G( \mathcal{R}) \rightarrow CH^*(\mathcal{R})
$$
and compatible with the cycle class maps:
\begin{equation*}
\xymatrix{CH^*(\mathcal{U}_C(n,d))\ar[d]_{cl}\ar[rr]^{\mu^{\ast}_{st}} && CH^*(\mathcal{R})\ar[d]^{cl} \\
\oplus_i H^{2i}(\mathcal{U}_C(n,d),\mathbb{Z})\ar[rr]_{~~~\mu^{\ast}_{coh}}&& \oplus_i H^{2i}(\mathcal{R},\mathbb{Z}) .
}
\end{equation*} 
We only need to note that the $\GL(N)$-equivariant subvariety $W^r_{\mathcal{R}}$ is represented by the class $[\widetilde{W^r_{n,d}}]$ in the equivariant cohomology of $\mathcal{R}$.

Hence the lemma is clear.

\end{proof} 

\begin{lemma}\label{nonempty}
The Brill-Noether loci  in $\mathcal{R}$ is non-empty (consequently in $U_C(n,d)$) if and only if its cohomology class is non-zero, in $H^*(\mathcal{R},\mathbb{Q})$ or in $H^*(\mathcal{U}_C(n,d),\mathbb{Q})$.
\end{lemma}
\begin{proof}
Clear.
\end{proof}

\section{Main theorems, when the rank is two}

We want to give some relations amongst the Brill-Noether loci in $\mathcal{U}_C(2,d)$.  In our context we fix degree $d$ to be $2(g-1)$. Denote $G:=\GL(N)$.

  In this case, Sundaram (cf. \cite{Su}) proved that $\widetilde{W^0_{2,2(g-1)}}$ is a divisor in ${U}_C(2,2(g-1))$.   We give some relations between the cohomology classes of the Brill-Noether loci in terms of cohomology class of $\widetilde{W^0_{2,2(g-1)}}$ in the moduli stack $\mathcal{U}_C(2,2g-2)$.
Since the moduli spaces ${U}_C(2,2(g-1))$ and ${SU}_C(2, \mathcal{L})$ are singular varieties,  Chow rings and cohomology rings of  the moduli stacks seem appropriate to consider. 

Consider the map $\pi_{\ast}:J^{ss}\subseteq J_{4(g-1)}(\widetilde{C}) \rightarrow {U}_{C}(2,2(g-1))$ as in (\ref{E39}).  Note that as we have taken $d=2(g-1)$, therefore it follows from (\ref{E3}) that $\delta=4(g-1)$.  Also from (\ref{E2}), we have 
\begin{equation*}
\delta=4(g-1)=\{4(g-1)+1 \}-1=\widetilde{g}-1.
\end{equation*}  
Hence we have the Theta divisor $\Theta:=W^0_{4(g-1)}(\widetilde{C})$ in $J_{4(g-1)}(\widetilde{C})$.  Following theorem says  that the Theta divisor of $\widetilde{C}$ intersects both $J^{ss}$ and its complement in $J_{4(g-1)}(\widetilde{C})$.

\begin{theorem}\label{T41}
	The Theta divisor of $J_{4(g-1)}(\widetilde{C})$, denoted by $\Theta$, does not lie inside the complement of $J^{ss}$ in $J_{4(g-1)}(\widetilde{C})$.  More precisely,
	\begin{enumerate}
		\item For any point $l\in J_{4(g-1)}(\widetilde{C})- \Theta$, $\pi_{\ast}(l)$ is semistable.
		\item There is a point $\xi \in \Theta$ such that $\pi_{\ast}(\xi)$ is semistable. 
	\end{enumerate}
\end{theorem}
\begin{proof}
	See \cite[Proposition 5.1; p. 176]{BNR}.
\end{proof}
Moreover we have that pullback of the divisor $\widetilde{W^0_{2,2(g-1)}}$ of ${U}_C(2,2(g-1))$ is the restriction of $\Theta$ to $J^{ss}$.

\begin{theorem}\label{T42}
	Let us denote the restriction of $\Theta$ to $J^{ss}$ by $\Theta \mid_{J^{ss}}$.  Then
	\begin{equation*}
	\pi_{\ast}^{-1}(\widetilde{W^0_{2,2(g-1)}})= \Theta \mid_{J^{ss}}.
	\end{equation*}
\end{theorem}
\begin{proof}
	See \cite[Lemma 6; p. 335]{BT}.  Also follows directly from the fact that $H^0(\widetilde{C},l)=H^0(C,\pi_{\ast}l)$.
\end{proof}

Now we want to check whether Theorem \ref{T41} and \ref{T42} hold for other Brill-Noether subvarieties of higher codimension.  By (\ref{LSS}), we have
$$
\pi_{\ast}^{-1}(\widetilde{W^r_{2,2(g-1)}})= W^r_{4(g-1)}\mid_{J^{ss}}.
$$

For our purpose,  we construct a scheme $S_0$ (cf. before Lemma \ref{S0}) to give relations between the cohomology classes Brill-Noether loci, on $\mathcal{U}_C(2,2(g-1))$.\\
\if
Let $\zeta \in J_{4(g-1)}(\widetilde{C})- J^{ss}$. Construct $S$ as blow up of $J_{4(g-1)}(\widetilde{C})$ at the point $\zeta$.  Then consider the following diagram.
\begin{equation*}
\label{E200}
\xymatrix{ & S\ar[ld]_{\pi_{\delta}} 
	\ar[dr]^{\pi_{{U}_C}}\\
	J_{\delta}(\widetilde{C}) \ar@{-->}[rr]_{\pi_{\ast}}&& {U}_C(2,d)}
\end{equation*}
where $\pi_{\delta}: S\rightarrow J_{\delta}(\widetilde{C})$ is a birational morphism.  Therefore by Theorem \ref{T2}, the restriction of the map $\pi_{\mathcal{U}_C}:S \rightarrow \mathcal{U}_C(2,d)$ to an open subset of $S$ is finite.
This diagram is commutative on the locus where the map $\pi_{\delta}$ and $\pi_{\ast}$ are defined.

We can construct a scheme $S$ with the following properties.
\begin{enumerate}
	\item $S$ is a smooth projective variety.
	 \item There is a morphism $q:S \rightarrow \mathcal{Q}$ such that  $\psi=\mu\circ q$, wherever $\mu$ is defined.
	\item There exists a birational morphism $\phi:S \rightarrow J_{4(g-1)}(\widetilde{C})$ and a generically finite morphism $\psi:S\rightarrow {U}_C(2,2(g-1))$.
	\item $\phi:\phi^{-1}(J^{ss})\rightarrow J^{ss}$ is an isomorphism.
	\item The following diagram is commutative.
		\begin{equation}
	\label{E199}
	\xymatrix{ & S \ar[ld]_{\phi} 
		\ar[dr]^{\psi}\\
		J_{4(g-1)}(\widetilde{C}) \ar@{-->}[rr]_{\pi_{\ast}}&& {U}_C(2,2(g-1))}
	\end{equation}
	Moreover this diagram is commutative whenever the domains of the involved rational maps are chosen properly.  In particular, we have 
	the following commutative diagram.
	\begin{equation*}
	\label{E200}
	\xymatrix{ &\phi^{-1}(J^{ss}) \ar[ld]_{\phi}^{\cong} 
		\ar[dr]^{\psi}\\
		J^{ss} \ar[rr]_{\pi_{\ast}}&& {U}_C(2,2(g-1))}
	\end{equation*}
\end{enumerate}

To construct $S$ as above, we use the  Poincar\'e bundle  $\mathcal{P}$ over $\widetilde{C}\times J_{4(g-1)}(\widetilde{C})$, and the morphism on the semistable points $J^{ss}\subset J_{4(g-1)}(\widetilde{C})$:
 $$
 J^{ss}\to \mathcal{Q}
 $$
  is induced by the family $((\pi\times id)_*\mathcal{P})|_{C\times J^{ss}}$.
Resolving this map and its singularities corresponds to a smooth projective variety $S$, with above properties.

Then we have the following diagram.

\begin{equation*}
\label{E201}
\xymatrix{ \tilde{C}\times S \ar[d]_{id \times \phi}\ar[rr]^{\pi \times id}&& C\times S \ar[d]_{id\times \psi}\ar[rr]&& S \ar[d]^{\psi}\\
	\widetilde{C}\times J_{4(g-1)}(\widetilde{C})\ar@{-->}[rr]_{\pi \times \pi_{\ast}}&& C \times {U}_C(2,2(g-1))\ar[rr]	&& {U}_C(2,2(g-1))}
\end{equation*}

Let us denote by $J^s$ the following set;
\begin{equation*}
J^{s}:=\{l\in J_{4(g-1)}(\widetilde{C})\mid \pi_{\ast}l\in {U}^s_{C}(2,2(g-1))\}.
\end{equation*}
Define $S_0:=\phi^{-1}(J^s)$. Then we have the following lemma:
\fi

Recall that we have the Poincar\'e bundle  $\mathcal{P}$ on $\widetilde{C}\times J_{4(g-1)}(\widetilde{C})$. Let $$\widetilde{v}_1:\widetilde{C}\times J_{4(g-1)}(\widetilde{C}) \to \widetilde{C},\,,\widetilde{v}_2:\widetilde{C}\times J_{4(g-1)}(\widetilde{C}) \to J_{4(g-1)}(\widetilde{C})$$ be the projections. Recall from Section \ref{sec:Brill-Noether loci on U} that we have fixed $m \gg 0$ such that any semistable vector bundle $E$ over $C$ of rank $n$ and degree $d$ is $m$-regular. We also now choose $m\gg 0$ so that $H^1(\widetilde{C}, L\otimes \pi^*\mathcal{O}(m))=0$ for any $L\in J_{4(g-1)}(\widetilde{C})$. Now consider 
the sheaf $$\mathcal{V}:=[\widetilde{v}_{2*}(\mathcal{P}\otimes \widetilde{v}^*_{1}\pi^*\mathcal{O}(m))]^{\vee}\,.$$ 
From cohomology and base change theorem, it follows that $\mathcal{V}$ is in fact a vector bundle. Now
consider the projective bundle 
$$\mathbb{P}:=\mathbb{P}(\mathcal{V}^{\oplus N})\;,$$ where $N$ is as in \eqref{equation for N} for $(n,d)=(2,2g-2)$, i.e, $N=2m$.  The scheme $\mathbb{P}$ parametrizes line bundles $L$ of degree $4(g-1)$ on $\widetilde{C}$ along with $N$ sections of $L\otimes \pi^*\mathcal{O}(m)$ upto scalars. This can be seen as follows: Let $$p_{\mathbb{P}}:\mathbb{P}\to J_{4(g-1)}(\widetilde{C})\,, \widetilde{v}_{1,\mathbb{P}}:\widetilde{C}\times \mathbb{P}\to \widetilde{C}\,, \widetilde{v}_{2,\mathbb{P}}:\widetilde{C}\times \mathbb{P}\to \mathbb{P}$$
be the projections. Over $\mathbb{P}$
we have the universal quotient 
$$p_{\mathbb{P}}^*\mathcal{V}^{\oplus N}\to \mathcal{O}_{\mathbb{P}}(1)\to 0\;.$$
Therefore we have a morphism between vector bundles over $\widetilde{C}\times \mathbb{P}$ :
\begin{align}\label{eqn-universal sections}
\widetilde v_{2,\mathbb{P}}^*\mathcal{O}_{\mathbb P}(-1) \to [\widetilde v_{2,\mathbb{P}}^*\mathcal{V}^{\vee}]^{\oplus N}
= & 
\widetilde v_{2,\mathbb{P}}^*p_{\mathbb{P}}^*[\widetilde{v}_{2*}(\mathcal{P}\otimes\widetilde v^*_{1}\pi^*\mathcal{O}(m))]^{\oplus N} \nonumber\\
= & ({\rm id}\times p_{\mathbb{P}})^*\widetilde{v}_2^*[\widetilde{v}_{2*}(\mathcal{P}\otimes \widetilde v^*_{1}\pi^*\mathcal{O}(m))]^{\oplus N} \nonumber\\
\to & ({\rm id}\times p_{\mathbb{P}})^*[\mathcal{P}\otimes \widetilde v^*_{1}\pi^*\mathcal{O}(m))]^{\oplus N} \nonumber\\
\cong & ({\rm id}\times p_{\mathbb{P}})^*\mathcal{P}^{\oplus N}\otimes \widetilde v^*_{1}\pi^*\mathcal{O}(m)\;.
\end{align}
Here the second equality follows from the following diagram, the second arrow is given by the adjunction morphism and the last isomorphism is given by projection formula.
\[
\begin{tikzcd}
    \widetilde{C}\times \mathbb{P} \ar[r,"\widetilde v_{2,\mathbb{P}}"] \ar[d,"{\rm id}\times p_{\mathbb{P}}"] & \mathbb{P} \ar[d,"p_{\mathbb{P}}"] \\
    \widetilde{C} \times J_{4(g-1)(\widetilde{C})} \ar[r,"\widetilde{v}_2"] & J_{4(g-1)}(\widetilde{C})
\end{tikzcd}
\]
Restricting the morphism (\ref{eqn-universal sections}) to  $\widetilde{C}\times x$ for $x\in \mathbb{P}$, we get $N$ sections of the line bundle (twisted by $\pi^*\mathcal{O}(m)$) corresponding to $p_{\mathbb{P}}(x)\in J_{4(g-1)}(\widetilde{C})$ and this produces a bijection between points of $\mathbb{P}$ and line bundles $L$ of degree $4(g-1)$ on $\widetilde{C}$ along with $N$ sections of $L\otimes \pi^*\mathcal{O}(m)$ upto scalars. Note that we also have an action of $GL(N)$ on $\mathbb{P}$ which permutes these sections. 

Let $$v_{1,\mathbb{P}}:C\times \mathbb{P}\to C\,, v_{2,\mathbb{P}}:C\times \mathbb{P}\to \mathbb{P}$$ be the projections. The morphism (\ref{eqn-universal sections}) can be thought of as $N$ section of the bundle 
$$({\rm id}\times p_{\mathbb{P}})^*\mathcal{P}\otimes \widetilde v^*_{1}\pi^*\mathcal{O}(m)\otimes \widetilde{v}_{2,\mathbb{P}}^*\mathcal{O}_{\mathbb{P}}(1)\,.$$
Applying $(\pi \times {\rm id})_*$ and projection formula we get $N$ sections of the following vector bundle on $C\times \mathbb{P}$:
$$(\pi \times {\rm id})_*[({\rm id}\times p_{\mathbb{P}})^*\mathcal{P}]\otimes  v^*_{1,\mathbb{P}}\mathcal{O}(m)\otimes v_{2,\mathbb{P}}^*\mathcal{O}_{\mathbb{P}}(1)\,,$$
i.e. we have a morphism 
\begin{equation}\label{eqn-direct image of universal sectins}
\mathcal{O}^N\to (\pi \times {\rm id})_*[({\rm id}\times p_{\mathbb{P}})^*\mathcal{P}]\otimes  v^*_{1,\mathbb{P}}\mathcal{O}(m)\otimes v_{2,\mathbb{P}}^*\mathcal{O}_{\mathbb{P}}(1)\,.
\end{equation}
Let $\mathbb{U}\subset \mathbb{P}$ be the open set such that (\ref{eqn-direct image of universal sectins}) 
restricted to $C\times \mathbb{U}$ is a surjection. Note that its points correspond to line bundles $L$ of degree $4(g-1)$ over $\widetilde{C}$ with $N$ sections of the $L\otimes \pi^*\mathcal{O}(m)$ such that the induced map $\mathcal{O}^N\to \pi_*L \otimes \mathcal{O}(m)$ is a surjection. Note that $\mathbb{U}$ is again $GL(N)$-equivariant. From the universal property of the Quot Scheme $\mathcal{Q}$ introduced in Section \ref{sec:Brill-Noether loci on U} which parametrizes quotients of $\mathcal{O}^N$ on $C$ of rank $2$ and degree $2(g-1)+2m$, we get a $GL(N)$-equivariant morphism 
$$\mathbb{U}\to \mathcal{Q}\,.$$
By taking an equivariant resolution of points of indeterminacy \cite[Theorem 1]{RY} of the rational map $\mathbb{P}\dashrightarrow{} \mathcal{Q}$ we get a smooth projecive variety $\mathbb{P}'$ with a morphism $\psi:\mathbb{P}'\to \mathcal{Q}$. Hence we have a diagram
\begin{equation}\label{eqn-blow up}
\begin{tikzcd}
    \mathbb{P}' \ar[d,"\phi"] \ar[dr,"\psi"] &  \\
    \mathbb{P} \ar[r, dashed] \ar[d,"p_{\mathbb{P}}"] & \mathcal{Q} \ar[d, dashed] \\
    J_{4(g-1)}(\widetilde{C}) \ar[r, dashed] & U_C(2,2(g-1))
\end{tikzcd}
\end{equation}
Here $\phi$ is birational and both $\phi$ and $\psi$ are $GL(N)$-equivariant. Since the rational map $\pi_{\ast}$ on $J_{4(g-1)}(\widetilde{C})$ is dominant (cf. Theorem \ref{T2}), the morphism
$\psi$ is surjective.

Define $S_0:=\psi^{-1}(\mathcal{R})$ and let $\psi_0:=\psi|_{S_0}$ and $\phi_0:=\phi|_{S_0}$ i.e. we have the diagram
\[
\begin{tikzcd}
    S_0 \ar[r, "\psi_0"] \ar[d, "\phi_0"] & \mathcal{R} \ar[d, hook]\\
    \mathbb{P}  & \mathcal{Q}
\end{tikzcd}
\]
Then we have the following lemma:
\begin{lemma}\label{S0}
$\psi_0^{-1}(W^r_{\mathcal{R},\mathcal{F}'(-m)})=(p_{\mathbb{P}}\circ \phi_0)^{-1}(W^r_{4(g-1)}(\widetilde{C}))\,.$
\end{lemma}
\begin{proof}
By Lemma \ref{functoriality of Brill-Noether loci} we have that
$\psi_0^{-1}(W^r_{\mathcal{R},\mathcal{F}'(-m)})= W^r_{S_0, \psi_0^*\mathcal{F}'(-m)}=W^r_{S_0,}$
and hence by the universal property of Quot schemes, it is the $r$-th Brill-Noether loci associated to the family 
$$(\pi \times {\rm id})_*[({\rm id}\times p_{\mathbb{P}})^*\mathcal{P}]\otimes v_{2,\mathbb{P}}^*\mathcal{O}_{\mathbb{P}}(1)=
(\pi \times {\rm id})_*[({\rm id}\times p_{\mathbb{P}})^*\mathcal{P}\otimes \widetilde v_{2,\mathbb{P}}^*\mathcal{O}_{\mathbb{P}}(1)]
.$$
By Lemma \ref{Brill Noether loci of pushforward of a bundle} this is same as the $r$-th Brill Noether locus associated to the line bundle 
$$({\rm id}\times p_{\mathbb{P}})^*\mathcal{P}\otimes \widetilde v_{2,\mathbb{P}}^*\mathcal{O}_{\mathbb{P}}(1)\;.$$
Now it follows from the definition of Brill-Noether loci that this locus is same as $W^r_{S_0, ({\rm id}\times p_{\mathbb{P}})^*\mathcal{P}}$ which again by Lemma \ref{functoriality of Brill-Noether loci} is same as 
pullback of $W^{r}_{4(g-1)}(\widetilde{C})$ via $p_{\mathbb{P}}\circ \phi_0$. This completes the proof of the lemma.
\end{proof}


\subsection{Poincar\'e type relations on moduli stacks}

Assume that $C$ is a general smooth projective curve and  $\tilde{C}\rightarrow C$ is a general smooth spectral curve, which is a double ramified covering of $C$.
We recall the cycle class maps between the Chow rings and cohomology rings
$$
CH^*(X,\mathbb{Q})\rightarrow H^*(X,\mathbb{Q})
$$
of the  moduli stacks $X=  \mathcal{U}_C(2,2(g-1))$ and $X=\mathcal{SU}_C(2, \mathcal{L})$ (for $\mathcal{L}\in J_{2(g-1)}(C)$) (cf. \S \ref{singmoduli}).

We start with the following lemma.

\begin{lemma}\label{descend}
The divisor classes on $J_{4(g-1)}(\widetilde{C}) $ descend to the moduli stack $\mathcal{U}_C(2,2(g-1))$ via diagram \eqref{eqn-blow up}.
\end{lemma}
\begin{proof}
Recall from \cite[Proposition 5.7]{BNR} a commutative diagram:
\[
\begin{tikzcd}
P'\times J_{g-1}(C)\ar[r]\ar[d, dashed]& J_{4(g-1)}(\widetilde{C})\ar[d,dashed, "\pi_{\ast}"] \\
{SU}_C(2) \times J_{g-1}(C) \ar[r] & {U}_C(2, 2(g-1)).
\end{tikzcd}
\]
Here $P'$ is the Prym variety associated to the covering $\widetilde{C}\rightarrow C$ and ${SU}_C(2)$ is the fixed determinant moduli space with determinant isomorphic to the trivial line bundle over $C$. The map $P'\times J_{g-1}(C)$ and $J_{4(g-1)}(\widetilde{C})$ is an isogeny and hence it induces isomorphisms on Picard groups tensored with $\mathbb{Q}$. In loc.cit it is shown that  the indeterminacy locus of the dominant rational map $\pi_*$ has codimension at least two and the same is true when restricted to $P'$. Therefore,  since the rational map $\pi_*: J_{4(g-1)}(\tilde{C}) \dashrightarrow U_C(2,2(g-1))$ factors via $\mathcal{U}_C(2,2(g-1))$, it is enough to show that any rational multiple of a divisor on $P'\times J_{g-1}(C)$ is a pullback of a divisor from $U_C(2,2(g-1))$.

Now since the Picard rank of the Jacobian of a general double cover $\widetilde{C}\to C$ is $2$ \cite[Corollary 5.3]{B}, we have that  Picard rank of $P'\times J_{g-1}(C)$ is $2$. Therefore the Neron-Severi space of $P'\times J_{g-1}(C)$ over $\mathbb Q$ is generated by pullbacks from divisors from $P'$ and $J_{g-1}(C)$. Since the Picard rank of the general Prym variety $P'$ is $1$ and the indeterminacy locus of $P'\to SU_C(2)$ has codimension $\geq 2$, its divisors descend to $SU_C(2)$ and hence the divisors of $P'\times J_{g-1}(C)$ descend to $SU_C(2)\times J_{g-1}(C)$. Now since the map $SU_C(2)\times J_{g-1}(C)\to U_C(2,2(g-1))$ is \'etale,(in particular finite) and both schemes have Picard rank $2$. This implies that their Neron-Severi space over $\mathbb{Q}$ are isomorphic. This shows that the divisors of $SU_C(2)\times J_{g-1}(C)$ descend to $U_C(2,2(g-1))$. This completes the proof of the lemma.

\if
The proof of loc.cit. implies that the polarisations on $P'$ and $J_{g-1}(C)$ descend on the moduli space ${U}_C(2,2(g-1))$. By functoriality, via the diagram \eqref{eqn-blow up}, the divisor classes descend on $\mathcal{U}_C(2,2(g-1))$.
\fi
\end{proof}

We now show the following.

\begin{theorem}\label{MT}
The cohomology class of a Brill-Noether subvariety on the moduli stack $\mathcal{U}_C(2,2(g-1))$ can be expressed as a polynomial on the divisor classes. In particular, the tautological algebra generated by the Brill-Noether loci is generated by the divisor classes.
\end{theorem}
\begin{proof}
Recall that we have the following diagram
\[
\begin{tikzcd}
     S_0 \ar[d,"p_{\mathbb{P}}\circ \phi_0"] \ar[r, "\psi_0"] & \mathcal{R} \ar[d] \\
    J_{4(g-1)}(\widetilde{C}) \ar[r,dashed] & \mathcal{U}_C(2,2(g-1))
\end{tikzcd}
\]
Note that $\psi_0$ is $G$-equivariant and since the morphism 
$\psi:\mathbb{P}'\to \mathcal{Q}$ is dominant and proper, $\psi_0$ is surjective. In particular we have an injection
$$H^*(\mathcal{U}_C(2,2(g-1)),\mathbb{Q})=H^*_G(\mathcal{R},\mathbb{Q})\to H^*_G(S_0,\mathbb{Q})$$
Moreover, by Lemma \ref{descend} we have the divisors classes of $J_{4(g-1)}(\widetilde{C})$ descend to the classes in $H^*(\mathcal{U}_C(2,2(g-1))$, it is enough to show that pullback of the Brill-Noether class $W^r_{\mathcal{R},\mathcal{F}'(-m)}$ to $S_0$ can be expressed as a polynomial on the pullback of  divisor classes from $J_{4(g-1)}(\widetilde{C})$ via $p_{\mathbb{P}}\circ \phi_0$. Now by Lemma \ref{S0} pullback of $W^r_{\mathcal{R},\mathcal{F}'(-m)}$ to $S_0$ is given by the class $(p_{\mathbb{P}}\circ \phi_0)^{-1}(W^r_{4(g-1)}(\widetilde{C}))$. Now by Theorem \ref{Bis}, the cohomology class of this Brill-Noether locus $W^r_{d}(\widetilde{C})\subset J_{4(g-1)}(\widetilde{C})$ is a polynomial expression on the divisor classes in $H^*(J_{4(g-1)}(\widetilde{C}))$. This completes the proof of the theorem.

\if
Recall that we have the  morphism $$p_{\mathbb{P}}\circ \phi_0:S_0\to J_{4(g-1)}\,.$$
By Theorem \ref{Bis}, the cohomology class of the Brill-Noether locus $W^r_{d}(\widetilde{C})\subset J_{4(g-1)}(\widetilde{C})$ is a polynomial expression on the divisor classes in $H^*(J_{4(g-1)}(\widetilde{C}))$. This implies that in $H^*(S,\mathbb{Q})$, the pullback of the cohomology class $[W^r_{d}(\widetilde{C})]$ is the cohomology class of  the Brill-Noether locus $W^r_{S_0}$ and it is a  polynomial expression on the pullback of the divisor classes on $J_{4(g-1)}(\widetilde{C})$.

Recall that we have the $G$-equivariant morphism $\psi_0:S_0\to \mathcal{R}$ and since the morphism $\mathbb{P}'\to \mathcal{Q}$ is proper and quasifinite, $\psi_0$ is also proper and quasifinite. Therefore the morphism
$$H^*(\mathcal{U}_C(2,2(g-1)),\mathbb{Q})=H^*_G(\mathcal{R},\mathbb{Q})\to H^*_G(S_0,\mathbb{Q})$$ is an injection.

Recall the morphism
$$
\phi:S \rightarrow J_{4(g-1)}(\widetilde{C}).
$$
Now $\phi$ is a birational morphism and let $E\subset S$ be the exceptional locus.
Hence, we have the following equality of cohomology rings:
$$
H^*( S,\mathbb{Q})= H^*(J_{4(g-1)}(\widetilde{C}),\mathbb{Q})\oplus H^*(E,\mathbb{Q}). 
$$

By Theorem \ref{Bis}, the cohomology class of the Brill-Noether locus $W^r_{d}(\widetilde{C})\subset J_{4(g-1)}(\widetilde{C})$ is a polynomial expression on the divisor classes in $H^*(J_{4(g-1)}(\widetilde{C}))$. This implies that in $H^*_{G}(S,\mathbb{Q})$, the pullback of the cohomology class $[W^r_{d}(\widetilde{C})]$ is the cohomology class of  the Brill-Noether locus $W^r_S \subset S$ and it is a  polynomial expression on the pullback of the divisor classes on $J_{4(g-1)}(\widetilde{C})$.

Recall that $S$ was constructed such that $q: S\rightarrow \mathcal{Q}$ and $\psi=\mu\circ q$, wherever $\mu$ is defined.

Denote $S':= q^{-1}(\mathcal{R})$.  Since $\mathcal{R}$ is a smooth variety  there are pullback maps on the Chow cohomologies, and using \eqref{liftchow}, we have:
$$
CH^*(\mathcal{U}_C(2,2(g-1)))=CH^*_G(\mathcal{R})\stackrel{\mu_{st}^*} {\rightarrow} CH^*(\mathcal{R})\stackrel{q^*}{ \rightarrow} CH^*(S').
$$
  By Lemma \ref{qchowcoh}, 
 \begin{equation}\label{BNfinalclass}
 q^*\mu_{st}^*[\widetilde{W^r_{2,2(g-1)}}]= q^* [W^r_{\mathcal{R}}]= [W^r_{S'}].
 \end{equation}
 
 Since $\mathcal{R}\subset \mathcal{Q}$ is an open subvariety of $\mathcal{Q}$, using the localization sequence
 $$
  CH^*(S) \rightarrow CH^*(S') \rightarrow 0,
  $$
  we deduce that $[W^r_S] \mapsto [W^r_{S'}]$.
  
  The above Chow cohomology diagram is compatible, via cycle class maps,  with the cohomology rings:
  $$
  H^* (\mathcal{U}_C(2,2g-2)),\mathbb{Q}) \stackrel{\mu_{st}^*}{\rightarrow}  H^*(\mathcal{R},\mathbb{Q})\stackrel{q^*}{ \rightarrow} H^*(S',\mathbb{Q})
 $$
 together with a restriction
 $$
 H^*(S,\mathbb{Q}) \stackrel{t}{\rightarrow} H^*(S',\mathbb{Q}).
 $$
  In particular, $[W^r_S]\mapsto [W^r_{S'}]$,  and it is a  polynomial expression on the pullback of the divisor classes on $J_{4(g-1)}(\widetilde{C})$.

Since the map $S'\rightarrow \mathcal{R}$,  is a proper, generically finite morphism, the  map on the  cohomologies is injective:
$$
\psi^*_{coh}:\,H^*(\mathcal{R},\mathbb{Q}) \hookrightarrow H^*(S',\mathbb{Q}).
$$ 

 Furthermore, by \eqref{BNfinalclass}, we deduce that 
\begin{equation}\label{equalclass}
\psi^*_{coh}[\widetilde{W^r_\mathcal{R}}]\,=\, [W^r_{S'}].
 \end{equation}

 By Lemma \ref{maincomp2} and Lemma \ref{descend},  we know that the divisor classes descend (equivalently they are $\GL(N)$-invariant classes). In other words, any polynomial expression on the divisor classes lies in the equivariant cohomology $H^*_G(\mathcal{R},\mathbb{Q})$. 
This implies that any polynomial expression on the divisor classes  descends on the cohomology of the moduli stack. 
 Now  \eqref{equalclass} implies that the cohomology class of the Brill-Noether locus in $H^*(\mathcal{U}_C(2,2(g-1)),\mathbb{Q})$  is expressible as a polynomial on the divisor classes. 

 \fi
\end{proof}

\subsection{Relations on the moduli stack $\mathcal{SU}_C(2,\mathcal{L})$}

Consider the determinant morphism
$$
\det: \mathcal{U}_C(2,2(g-1)) \rightarrow J_{2(g-1)}(C).
$$
The inverse image ${\det}^{-1}(\mathcal{L})$ is the sub(moduli) stack $\mathcal{SU}_C(2,\mathcal{L})$, for a line bundle $\mathcal{L}$ on $C$ of degree $2(g-1)$. 
 Denote the Brill-Noether locus
\begin{equation}\label{definition-bn loci over moduli space with fixed determinant}
\widetilde{W^{r,\mathcal{L}}_{2,2(g-1)}}\,:= \,   \widetilde{W^r_{2,2(g-1)}} \cap \mathcal{SU}_C(2,\mathcal{L})\,.
\end{equation}
\begin{corollary}\label{MT2}
The cohomology class of a Brill-Noether locus $\widetilde{W^{r,\mathcal{L}}_{2,2(g-1)}}$, as in \eqref{definition-bn loci over moduli space with fixed determinant}, in the moduli stack $\mathcal{SU}_C(2,\mathcal{L})$ is a polynomial expression on the class of the Theta divisor, with rational coefficients. In particular the tautological algebra is generated by the class of the Theta divisor $\Theta$.

\end{corollary}
\begin{proof}
 Consider the inclusion:
$$
j: \mathcal{SU}_C(2, \mathcal{L}) \hookrightarrow \mathcal{U}_C(2,2(g-1)).
$$
The pullback map on the cohomology ring
$$
j^*: H^*(\mathcal{U}_C(2,2(g-1),\mathbb{Q}) \rightarrow  H^*(\mathcal{SU}_C(2,\mathcal{L}),\mathbb{Q})
$$ 
is a ring homomorphism. By Theorem \ref{MT},  the cohomology class of  the Brill-Noether locus is a polynomial expression on the divisor classes on $\mathcal{SU}_C(2,\mathcal{L})$. 
Using Proposition \ref{Picequiv}, we know that the Picard group of $\mathcal{SU}_C(2,\mathcal{L})$ is generated by the Theta divisor $\Theta$. This gives the relation, for any irreducible component:
$$
[\widetilde{W^{r,\mathcal{L}}_{2,2(g-1)}}]\, =\, \alpha. [\Theta]^{t(r)} \,\in\, H^*(\mathcal{SU}_C(2,\mathcal{L}),\mathbb{Q})
$$
for some $\alpha \in \mathbb{Q}$ and $t(r)$ is the codimension of an irreducible component of the Brill-Noether locus.

\end{proof}

\section*{Acknowledgements}  This is a part of the doctoral work of the third named author AM, and was done during his visits in 2019 at IMSc.  He was supported by the `Complex and Algebraic geometry' DAE-project at IMSc and by University Grants Commission (UGC) (ID - 424860) at University of Hyderabad.  Currently, he is supported by Indian Institute of Technology Madras (Office order No.F.ARU/R10/IPDF/2024).  The authors thank D. Arapura for pointing out his work (cf. \cite{Ar}), and for useful communications. The authors are grateful to P. Newstead for his remarks, for sharing his thoughts on the previous versions,  pointing out errors/mistakes and pointing out his related works.  The authors are grateful to the anonymous referee for many valuable suggestions, which have greatly improved the content and presentation of the paper.

\section*{Declaration of Interests}
The authors declare that there is no potential competing interest.

\end{document}